\documentclass[12pt]{amsart}
\usepackage{amsthm}
\usepackage{amssymb}
\usepackage{amsmath}
\usepackage{graphicx, enumerate}
\usepackage{xcolor, soul}
\usepackage[all]{xy}
\usepackage{epsfig} 

\usepackage[el,nf]{coelacanth}

\usepackage{graphicx}
\usepackage{verbatim}
\usepackage{epsfig}
\usepackage{wasysym}
\usepackage{url}
\usepackage{mathrsfs}
\usepackage{bbm}
\usepackage{nicefrac}
\usepackage[colorlinks=true,linkcolor=blue,urlcolor=blue,citecolor=blue]{
hyperref}

\def\barB{\overline{B}}

\def\GP{\mathcal{GP}}

\def\A{\mathcal{A}}

\def\H{\mathcal{H}}
\def\M{\mathcal{M}}

\def\C{\mathbb{C}}
\def\Q{\mathbb{Q}}

\def\N{\mathbb{N}}
\def\R{\mathbb{R}}

\def\supp{{\rm supp}}

\DeclareFontFamily{U}{mathx}{\hyphenchar\font45}
\DeclareFontShape{U}{mathx}{m}{n}{
      <5> <6> <7> <8> <9> <10>
      <10.95> <12> <14.4> <17.28> <20.74> <24.88>
      mathx10
      }{}
\DeclareSymbolFont{mathx}{U}{mathx}{m}{n}

\DeclareMathAccent{\widecheck}{0}{mathx}{"71}

\def\Ext{\mbox{\rm Ext}_\R}

\newtheorem{thm}{Theorem}[section]

\newtheorem{lem}[thm]{Lemma}
\newtheorem{prop}[thm]{Proposition}
\newtheorem{exa}[thm]{Example}

\theoremstyle{definition}
\newtheorem{rem}[thm]{Remark}
\newtheorem{notation}[thm]{Notation}

\newcommand{\tc}[2]{{\pi_{#2}({#1})}}
\newcommand{\tcm}[2]{{\pi^{#2}({#1})}}

\renewcommand{\emph}{\textbf}

\usepackage[foot]{amsaddr}

\title{New insights into Gleason parts for an algebra of holomorphic functions}

\author[Daniel Carando]{Daniel Carando$^1$ }
\address{$^1$ \textit{Corresponding author}. Departamento de Matem\'{a}tica, Facultad de Ciencias Exactas y Naturales, Universidad de Buenos Aires, (1428) Buenos Aires,
Argentina and IMAS (UBA - CONICET). ORCID 0000-0002-5519-8697} \email{dcarando@dm.uba.ar}  

\author[Ver\'onica Dimant]{Ver\'onica Dimant$^2$}
\address{$^2$Departamento de Matem\'{a}tica y Ciencias, Universidad de San
Andr\'{e}s, Vito Dumas 284, (B1644BID) Victoria, Buenos Aires,
Argentina and CONICET. ORCID 0000-0003-2277-0551} \email{vero@udesa.edu.ar}

\author[Tom\'as Rodr\'{i}guez]{Jorge Tom\'as Rodr\'{i}guez$^3$}
\address{$^3$Departamento de Matem\'{a}tica and NUCOMPA, Facultad de Cs. Exactas, Universidad Nacional del Centro de la Provincia de Buenos Aires, (7000) Tandil, Argentina and CONICET. ORCID 0000-0003-4693-2498}
\email{jtrodriguez@nucompa.exa.unicen.edu.ar}

\begin{document}

\begin{abstract}
    We study the structure of the spectrum of the algebra of uniformly continuous holomorphic functions on the unit ball of $\ell_p$. Our main focus is the relationship between \emph{Gleason parts} and \emph{fibers}. For every $z \in B_{\ell_p}$ with $1 < p < \infty$, we prove that the fiber over $z$  contains $2^{\mathfrak{c}}$ distinct Gleason parts. We also investigate some of the properties of these Gleason parts and show the existence of many strong boundary points in certain fibers. We then examine the case $p = 1$, where similar results on the abundance of Gleason parts within the fibers hold, although the arguments required are more involved. Our results extend and complete earlier work on the subject, providing answers to previously posed questions. 
\end{abstract}


\maketitle

\delimiterfactor=500
\delimitershortfall=25pt

\section{Introduction}

The thorough study of the maximal ideal space (or spectrum) of an algebra of holomorphic functions has its roots in the seminal work of I. J. Schark \cite{schark1961maximal}. There, the idea of restricting each homomorphism to the underlying space was proposed, leading to the useful notion of \emph{fiber}. Earlier, a relevant tool for understanding the analytic structure of the spectrum of a uniform algebra was the concept of \emph{part}, introduced by Gleason \cite{gleason1957} and much developed since then. Fibers and Gleason parts thus provide two different partitions of the spectrum of a holomorphic algebra, and our interest is to clarify their interaction.

The focus on algebras of holomorphic functions over \emph{infinite}-dimensional spaces began with the foundational work of Aron, Cole and Gamelin \cite{aron1991spectra}, inducing a large number of interesting results where fibers play a crucial role (see, for instance, \cite{aron2016cluster, aron2018analytic, aron1996regularity, CaGaMaSeSpectra, cole1992analytic, farmer1998fibers}). Gleason parts for these types of algebras (and their relationship to fibers) were considered much later \cite{aron2020gleason,DimLasMae} and it is our intention here to continue this path for the spectrum of the algebra of uniformly continuous holomorphic functions on the unit ball of $\ell_p$ ($1 \leq p < \infty$). To precisely describe our goals, we need first to introduce some notation and definitions.

Let $X$ be a complex Banach space and \( \A_u(B_X) \) denote the Banach algebra of holomorphic and uniformly continuous functions on its open unit ball $B_X$, endowed with the supremum norm. Our object of study is the spectrum of this algebra, \( \M(\A_u(B_X)) \), which consists of all non-zero multiplicative linear functionals on \( \A_u(B_X) \). Its structure is intricately connected to the geometry of \( X \) and the properties of \( \A_u(B_X) \). 
The algebra \( \A_u(B_X) \) clearly contains $X'$, the dual of $X$. Thus, for each $z\in\overline B_{X''}$ (the closed unit ball of the bidual of $X$) we can define the fiber over $z$, denoted by \( \M_z \), as the set of all homomorphisms $\varphi\in \M(\A_u(B_X))$ such that $\varphi|_{X'}=z$. For $X=c_0$, it is known that each fiber is a singleton, while for $X=\ell_p$ fibers over interior points are thick (i.e., non singleton); and it is even so for some edge fibers in the case $p=1$.

The article \cite{aron2020gleason} presents some general results on the interaction between fibers and Gleason parts for the maximal ideal space of the algebras \( \A_u(B_X) \) and $\mathcal H^\infty(B_X)$ (bounded holomorphic functions on $B_X$), with a particular focus on the case $X=c_0$. The subsequent work \cite{DimLasMae} faces the same type of study for \( \M(\A_u(B_{\ell_p})) \),  revealing some interesting insights about the Gleason-fiber structure but also leaving many relevant open questions. 
In this article, we continue with this line, trying to solve (or produce significant progress toward) these open problems. Loosely speaking, one question that guides our work is the following: do thick fibers intersect (or contain) many Gleason parts? We are particularly interested in finding fibers that intersect as many Gleason parts as possible (as we shall see, this means $2^{\mathfrak c}$ Gleason parts). It is worth noting that although our focus is on \( \A_u(B_X) \), our two main results, Theorems \ref{thm:resultado-ppal-ellp} and \ref{thm:PG-puntos-interiores-ele-1}, can be extended to \( \mathcal{H}^\infty(B_X) \) through the natural projection between their spectra (see, for instance, \cite[Sect. 3]{aron2020gleason}).

The paper is organised as follows. In Section 2, we introduce the basic definitions, recall some known facts and establish several preliminary results. Section 3 is devoted to the study of the space $\A_u(B_{\ell_p})$ for $1 < p < \infty$. We show in Theorem \ref{thm:resultado-ppal-ellp} that, for any $z\in B_{\ell_p}$, the fiber $\M_z$ reaches \( 2^{\mathfrak{c}} \) different Gleason parts, extending the results obtained in \cite{DimLasMae} for integer values of $p$. Additionally, other properties of these Gleason parts are exhibited, and the existence of strong boundary points (see definition in Section \ref{subs: sbp}) in some fibers is proved, moving one step forward to the answer of another open question from \cite{DimLasMae}. In the final section, we address the case $p = 1$, showing the existence of \( 2^{\mathfrak{c}} \) different Gleason parts in $\M_z$, for each $z\in B_{\ell_1''}$ (see Theorem \ref{thm:PG-puntos-interiores-ele-1}). Although the statement is analogous to that of the previous section, we need to develop more involved arguments to prove it. Finally, we present in Proposition \ref{prop:enelborde} and Example \ref{ex-promedio} a subset of the unit sphere of $\ell_1''$ for which the corresponding fibers intersect different Gleason parts.



\section{Preliminaries and general results}

For an infinite-dimensional complex Banach space $X$ we denote by $X'$ and $X''$  its dual and bidual and use classical notations   $B_X$ for its open unit ball, $S_X$ for its unit sphere and $\overline B_X$ for its closed unit ball. 
A function $P:X\to \C$ is a \emph{continuous $m$-homogeneous polynomial} if there exists a (unique) continuous symmetric $m$-linear mapping $\widecheck{P}$ such that $P(x) = \widecheck{P}(x,\dots, x)$. We write $\mathcal{P}(^mX)$ for the space of continuous  $m-$homogeneous polynomials. 

Given an open set $U\subseteq X$, a mapping $f:U\to \C$ is said to be \emph{holomorphic} if for every $x_0\in U$ there exists a sequence $(P_mf(x_0))$, with each $P_mf(x_0)$ a continuous $m$-homogeneous polynomial, such that the series
\[
f(x)=\sum_{m=0}^\infty P_mf(x_0)(x-x_0)
\] converges uniformly in some neighbourhood of $x_0$ contained in $U$.
The space
\[
\A_u(B_X)=\{f:B_X\to\mathbb C:\  f \textrm{ is a uniformly continuous holomorphic function}\}
\] is a Banach algebra when endowed with the norm $\|f\|=\sup_{x\in B_X}|f(x)|$. Note that each function $f\in \A_u(B_X)$ extends continuously to $\overline B_X$.
The \emph{maximal ideal space} or \emph{spectrum}  of this algebra is the set
$$
\M(\A_u(B_X))=\{\varphi\in \A_u(B_X)':\ \varphi \textrm{ is multiplicative}\}\setminus\{0\}.
$$

In the spectrum $\M(\A_u(B_X))$, which is a subset of the unit sphere of $\A_u(B_X)'$, we consider the $w^*$-topology inherited from $\A_u(B_X)'$. With this topology, $\M(\A_u(B_X))$ is a compact set.

A typical strategy to produce large sets within the spectrum is via interpolating sequences. A sequence $\{\varphi_j\}\subset \M(\A_u(B_X))$ is \emph{interpolating} if given  $(\alpha_j)\in\ell_\infty$ there is $f\in \A_u(B_X)$ such that $\varphi_j(f)=\alpha_j$ for all $j$. When $\{\varphi_j\}$ is interpolating, its $w^*$-closure is identified with $\beta(\N)$, the Stone-\v{C}ech compactification of $\N$. Recall that the cardinality of  $\beta(\N)$ is $2^\mathfrak c$. Note also that for separable $X$, $\A_u(B_X)$ has cardinality $\mathfrak{c}$ (it is a subset of the set of all complex-valued continuous functions on $B_X$) and, then, the cardinality of the spectrum $\M(\A_u(B_X))$ is at most  $2^\mathfrak c$. Therefore, subsets built from interpolating sequences have maximal cardinality. 

Now we recall the definitions and basic known results about our target structures: fibers and Gleason parts. 

\subsection*{\textsc{Fibers}} Since the dual $X'$ is naturally embedded in the Banach algebra $\A_u(B_X)$, we can define a surjective linear map $\pi: \M(\A_u(B_X)) \to \barB_{X''}$ by $\pi(\varphi) = \varphi|_{X'}$. 
Then, the \emph{fiber} of $\M(\A_u(B_X))$ corresponding to a point $z \in \barB_{X''}$ (or the \emph{fiber over $z$}) is the subset $\M_z\subset \M(\A_u(B_X)) $ given by
\begin{equation*}
\M_z = \{\varphi \in \M(\A_u(B_X)) : \pi(\varphi) = z\} = \pi^{-1}(z).
\end{equation*}
For each $z\in \overline B_{X''}$, there is a distinguished element in the fiber $\M_z$, the \emph{evaluation homomorphism} $\delta_z$. This mapping  is defined as $\delta_z(f)=f(z)$ if $z\in X$ and $\delta_z(f)=\widetilde f(z)$ if $z\in X''$, where $\widetilde f$ is the canonical extension to the bidual (see \cite{AroBer78}). 
With this, we obtain the following commutative diagram, where $I_{X''}$ denotes the identity on $X''$:
$$
\xymatrix{
\overline B_{X''} \ar[r]^{\hspace{-.5cm}\delta\,\,\,} \ar[rd]_{I_{X''}}  & \M(\A_u(B_X))\ar[d]^{\pi}\\
& \barB_{X''}.\\
}
$$

\subsection*{\textsc{Gleason parts}} Another distinguished structure inside the spectrum is the concept of \emph{Gleason part}. First, recall the \emph{Gleason distance}, which is just the metric on the spectrum inherited from $\A_u(B_X)'$. That is,
for $\varphi, \psi \in \M(\A_u(B_X))$ the distance between them is
$$
\|\varphi - \psi\| = \sup\{|\varphi(f) - \psi(f)| : \|f\| \leq 1\}.
$$
Clearly, $\|\varphi - \psi\| \leq 2$ for any pair $\varphi, \psi$. Moreover, the condition $\|\varphi - \psi\| < 2$ defines an equivalence relation on $\M(\A_u(B_X))$, whose equivalence classes are referred to as Gleason parts. Thus, for $\varphi\in \M(\A_u(B_X))$, the Gleason part containing $\varphi$ is the set
\begin{equation*}
\GP(\varphi) = \{\psi \in \M(\A_u(B_X)) : \|\varphi - \psi\| < 2\}.
\end{equation*}

The \emph{pseudo-hyperbolic distance} $\rho$ is often employed to determine whether two elements belong to the same Gleason part. This is defined as
$$
\rho(\varphi, \psi) = \sup\{|\varphi(f)| : f \in \A_u(B_X), \|f\| \leq 1, \psi(f) = 0\}.
$$
It is a well-known result that $\|\varphi - \psi\| < 2$  if and only if $\rho(\varphi, \psi) < 1$. Therefore, the Gleason part of $\varphi$ can also be expressed as
\begin{equation*}
\GP(\varphi) = \{\psi \in \M(\A_u(B_X)) : \rho(\varphi, \psi) < 1\}.
\end{equation*}

We refer the reader to \cite{dineen2012complex} for the general theory of holomorphic functions on infinite-dimensional spaces, to \cite{stout1971theory}, \cite{gamelin2005uniform} or \cite{garnett2006bounded} for classical texts on uniform algebras and to \cite{CaGaMaSeSpectra} for a survey on spectra of algebras of holomorphic functions on Banach spaces.

\subsection*{\textsc{General results}} 
The following fact is useful in deciding whether two  elements belong to the same Gleason part or not:
two homomorphisms that share the Gleason part must approach $1$ in the same sequences of elements of the ball of $\A_u(B_X)$. More precisely, from \cite[Thm. 3.9]{stout1971theory} (see also \cite[Lem. 2.1]{DimLasMae}), we know that for $\varphi, \psi \in \M(\A_u(B_X))$, the equality $\GP(\varphi) = \GP(\psi)$ is equivalent to the following: if $\varphi(f_j) \to 1$ for some sequence $(f_j)_j \subset \overline{B}_{\A_u(B_X)}$, then $\psi(f_j) \to 1$. We will make frequent use of this result throughout the article.

\begin{rem}\label{rmk: PG-fibras}
   As mentioned above, fibers and Gleason parts are two different ways of partitioning the spectrum. Two basic aspects regarding their relationship are the following \cite[Prop. 1.1]{aron2020gleason}:
    \begin{itemize}
        \item $\GP(\delta_0)=\GP(\delta_z)$ for all $z\in B_{X''}$.
        \item $\GP(\varphi)\not=\GP(\psi)$ whenever $\varphi\in\M_z$, $\psi\in\M_\omega$ 
         and $z \in B_{X''}$, $\omega\in S_{X''}$.
    \end{itemize}
\end{rem}

The spectrum of $\A_u(B_X)$ inherits some properties through its connection with the spectrum of $\H_b(X)$, the Fr\'echet algebra of entire functions that are bounded on bounded subsets of $X$ endowed with the topology of uniform convergence on these bounded sets (see \cite[Sect. 6.3]{dineen2012complex}). The relationship between $\M(\H_b(X))$ and $\M(\A_u(B_X))$ is given via the \emph{radius function} \cite{aron2016cluster, aron1991spectra}, which for $\varphi \in \M(\H_b(X))$
is defined as
\begin{equation} \label{eq-radio}
R(\varphi)=\inf\big\{r>0\,\colon \, |\varphi(f)|\le \|f\|_{rB_X} \text{ for every } f \in \H_b(X)\big\}
\end{equation} where $\|f\|_{rB_X}=\sup\{|f(x)|\colon \|x\|<r\}$. From \cite[Lem.~3.1 and  3.2]{aron1991spectra} we know  that $R(\delta_z)=\|z\|$ for all $z\in\barB_{X''}$  and that $\|\varphi|_{X'}\| \le R(\varphi)$.

The aforementioned  connection between the spectra of $\H_b(X)$ and $\A_u(B_X)$ is given by the mapping
\begin{eqnarray*}
\Phi\colon \M(\A_u(B_X)) &\to &\big\{\varphi\in \mathcal{M}(\H_b(X)),\, R(\varphi)\leq 1\big\}\\
\varphi &\mapsto & [f\mapsto \varphi(f|_{B_X})],
\end{eqnarray*}
which is an onto homeomorphism (see  
\cite[Lem.~1.2]{aron2016cluster} or \cite[Thm.~2.5]{aron1991spectra}). This ensures that no misunderstandings arise when  defining  $R(\varphi)$ for $\varphi\in \M(\A_u(B_X))$. Moreover, since $\H_b(X)$ is dense in $\A_u(B_X)$,  for $\varphi\in \M(\A_u(B_X))$ we can compute $R(\varphi)$ using  \eqref{eq-radio} with $f$  in $\A_u(B_X)$ instead of $\H_b(X)$.

Gleason parts containing only one element are known as \emph{trivial} or \emph{singleton parts}, while Gleason parts with more than one element are called \emph{thick}. 
The study of Gleason parts for the spectrum of $\A_u(B_X)$ in \cite{aron2020gleason,DimLasMae} shows that $\GP(\delta_0)$ is thick and can contain not only all $\delta_z$ with $\|z\|<1$ but also copies of a ball inside the fiber of 0 (or other fibers). One reason for this behaviour is exposed in the following interaction between $\GP(\delta_0)$ and the radius function.

\begin{lem}\label{lem:radio menor a 1}
    Let $\varphi\in \M(\A_u(B_X))$ with $R(\varphi)<1$. Then, $\varphi\in\GP(\delta_0)$.
\end{lem}

\begin{proof}
We know that $|\varphi(f)|\le \|f\|_{R(\varphi)B_X}$ \cite[Lem. 2.1]{aron1991spectra} (actually, the referred lemma shows the inequality for $f\in \H_b(X)$ but the same argument holds for  $f\in\A_u(B_X)$).
Then, by Schwarz lemma,
$$
\rho(\varphi, \delta_0) =
\sup\{|\varphi(f)|:\  f \in \A_u(B_X), \|f\| \leq 1, f(0) = 0 \}\le R(\varphi) <1.
$$
Thus,  $\varphi\in\GP(\delta_0)$.
\end{proof}

We do not know whether the converse of the previous result holds. 

The following lemma will be used in the sequel to deal with homomorphisms that are limits of evaluations, allowing us to modify the net of evaluations to satisfy particular requirements.  Since it will be applied only in the reflexive case, we state and prove it for elements in  $X$, though it extends  naturally to the bidual.

\begin{lem}\label{lem:redes_cerca}
Let $\{z_\alpha\}_\alpha$ and $\{\widetilde z_\alpha\}_\alpha$ be two nets in $  \overline{B}_X$  such that $$\|z_\alpha-\widetilde z_\alpha\|\underset{\alpha}{\longrightarrow} 0.$$ If $(\delta_{z_\alpha})$ $w^*$-converges to $ \varphi \in \M(\A_u(B_{X}))$, then we also have $\varphi = w^*-\lim\delta_{\widetilde z_\alpha} $.
\end{lem}

\begin{proof}
    Given $f\in \A_u(B_{X})$, the uniform continuity of $f$ implies that $|f(z_\alpha)-f(\widetilde z_\alpha)|\underset{\alpha}{\longrightarrow} 0$. On the other hand, we know that $f(z_\alpha) \to \varphi(f)$ and this altogether gives $f(\widetilde z_\alpha) \to \varphi(f)$. Given that $f$ is arbitrary, we have $\varphi = w^*-\lim\delta_{\widetilde z_\alpha} $.
\end{proof}

{Recall that \cite[Prop. 1.2]{aron2020gleason}  says that if $\varphi=w^*-\lim \delta_{z_\alpha}$ with $\|z_\alpha\|\le r<1$, then $\varphi\in\GP(\delta_0)$. Since such a homomorphism $\varphi$ satisfies $R(\varphi)\le r$,  the conclusion can also be obtained as a consequence of our Lemma \ref{lem:radio menor a 1}. In light of this and Lemma \ref{lem:redes_cerca}, we obtain the following: if $\varphi=w^*-\lim \delta_{z_\alpha}$ with $\{z_\alpha\}_\alpha\subset  \overline{B}_X$ and $\varphi\not\in\GP(\delta_0)$ then there exists a net  $\{\widetilde z_\alpha\}_\alpha\subset S_X$ such that $\varphi=w^*-\lim \delta_{\widetilde z_\alpha}$. }

We end this section with some elementary results about continuous \( k \)-homogeneous polynomials that will be useful later on. The first is the observation that every \( k \)-homogeneous polynomial is Lipschitz on bounded sets. Moreover, for each \( k \in \mathbb{N} \), there exists a constant \( L(k) > 0 \), depending only on \( k \), such that the restriction of any norm-one continuous \( k \)-homogeneous polynomial \( P \) to the ball of radius \( r \) is Lipschitz with constant \( r^k L(k) \). That is,
\begin{equation}\label{eqn:Lipschitz}
    |P(x)-P(y)| \le r^k L(k) \|x-y\| \quad \text{for all }  x,y\in rB_X,\ P\in\mathcal P(^kX) \text{ with }\|P\|=1.
 \end{equation}
The second result concerning polynomials is \cite[Lem.~2.3]{carando2013lower}. For the reader's convenience, we state it below in the context of $\ell_p$ spaces.
\begin{lem}\label{lem:Pinasco}
Let \( 1 \leq p < \infty \), and let \( R, Q : \ell_p \to \mathbb{C} \) be homogeneous polynomials of degrees \( l \) and \( k \), respectively. Suppose that the sets of coordinates on which \( R \) and \( Q \) depend are disjoint. Then, the product \( RQ \) satisfies
\[
\|RQ\| = \left( \frac{l^l\, k^k }{(l+k)^{l+k}} \right)^{1/p} \|R\|\, \|Q\|.
\]
\end{lem}



\section{The case $1<p<\infty$}

Since \(\ell_p\) is uniformly convex  (\(1 < p < \infty\)), we know from \cite[Lem. 2.4]{aron2012cluster} that for every \(z \in S_{\ell_p}\), the fiber \(\M_z\) in the spectrum \(\M(\A_u(B_{\ell_p}))\) is the singleton \(\{\delta_z\}\). Moreover, by \cite[Prop. 4.1]{farmer1998fibers}, the Gleason part of \(\delta_z\) is also trivial, that is \(\M_z = \{\delta_z\} = \GP(\delta_z)\) (see also \cite[Prop. 3.1]{DimLasMae}).
 On the other hand, fibers over points $z\in B_{\ell_p}$ are thick, as shown for instance in \cite[Thm. 3.1]{aron2018analytic}. Thus, the search of several Gleason parts intersecting the same fiber makes sense for inner points. In this direction, in \cite[Thm. 3.3]{DimLasMae} is proved that the fiber   over 0 intersects $2^{\mathfrak c}$ different Gleason parts. 
 For non-zero vectors  $z\in B_{\ell_p}$, the corresponding result is obtained under the additional assumption that $p$ is an integer \cite[Prop. 3.5]{DimLasMae}. A natural question, posed as Open problem 4 in the same article, is whether this also holds when $p$ is not an integer. The main goal of this section is to answer this question affirmatively. For this, we use a different approach that focuses not on the value of $p$, but rather on the norm of $z$. Moreover, our construction provides additional structural insight. We demonstrate that each fiber \emph{contains}, rather than merely intersects,  $2^{\mathfrak c}$ different  Gleason parts, as stated in the main theorem of this section.

Throughout this and the following sections, we frequently work with  truncations of sequences, so we now fix a notation for this.
\begin{notation}\label{not:truncados}
    Given a sequence $z=(z(i))_{i\in \N}$, we write $\tc{z}{m}$ for the \emph{$m$-truncation of $z$} and $\tcm{z}{m}=z-\tc{z}{m}$ for the \emph{$m$-tail of $z$}. That is, 
    $$\tc{z}{m}=(z(1),\ldots, z(m), 0, 0,\ldots );$$
    $$\tcm{z}{m}=(0,\ldots, 0, z(m+1), z(m+2),\ldots ).$$
\end{notation}
Also, for a net $(\varphi_\alpha)_{\alpha\in I}\subset \M(\A_u(B_X))$, we write $$w^*-ac \{\varphi_\alpha : \alpha \in I \}$$ for the set of all  $w^{*}$-accumulation points of $(\varphi_\alpha)_{\alpha\in I}$.

With these notations at hand, we can state our first main result. 

\begin{thm}\label{thm:resultado-ppal-ellp}
Let $1<p<\infty$ and  $z\in B_{\ell_p}$. Then, there are $2^{\mathfrak c}$ different Gleason parts of $\M(\A_u(B_{\ell_p}))$ completely contained in  $\M_z$. 
More precisely, denoting $s=\left(1-\|z\|^p\right)^{1/p}$, the cardinality of 
$$w^*-ac\{ \delta_{\tc{z}{n}+s e_{n+1}} : n\in \mathbb{N}\}\subset \M_z$$
is $2^{\mathfrak c}$ and each $\varphi$ in this set belongs to a different Gleason part. Moreover,   $\GP(\varphi)\subset \M_z$ for every such $\varphi$.
\end{thm}

Furthermore, in Proposition \ref{prop:nocorona} below, we give a precise description of the $w^*$-limit of evaluations belonging to  Gleason parts over elements in $w^*-ac\{ \delta_{\tc{z}{n}+s e_{n+1}} : n\in \mathbb{N}\}$. We recall that the \emph{corona} consists of all homomorphisms that are not $w^*$-limit of evaluations; which, in general, is not known whether it is empty or not (see \cite{carleson1962interpolations} for the first positive result in this direction). Therefore, we are describing the elements in some Gleason part that are outside the corona.

We devote most of this section to the proof of Theorem \ref{thm:resultado-ppal-ellp}, which is developed in four main steps. First, in Proposition \ref{prop:distinta-gleason} we deal with fibers over elements of finite support $z\in B_{\ell_p}$ satisfying $\|z\|^p\in\Q$. Then, in Proposition \ref{prop:ele-p-no-racional}, we show that for an arbitrary $z\in B_{\ell_p}$, there exists
an element in the fiber over $z$ which does not belong to the Gleason part of $\delta_z$. By a somehow canonical argument (which turns out to be rather technical in this case),  we show in Proposition \ref{prop:resultado-ppal-ellp} that there are actually $2^{\mathfrak c}$ different Gleason parts intersecting the same fiber, thus answering the question posed in \cite{DimLasMae}. Finally, in Proposition \ref{prop:gleason-no-intersec}, we prove that those Gleason parts are in fact contained in the fiber of $z$. Let us start with the first step towards the proof of Theorem \ref{thm:resultado-ppal-ellp}.

\begin{prop}\label{prop:distinta-gleason}
    Let $1<p<\infty$, $m\in\N$, $z\in B_{\ell_p}$ with  $\supp( z)\subset \{1,\ldots, m\}$  and $\|z\|^p\in\Q$, and let $s=\left(1-\|z\|^p\right)^{1/p}$. Then, for any
    $\varphi\in w^*-ac\{ \delta_{z+s e_{n}} : n\in \mathbb{N}, n> m\}\subset\M_z$
    we have  $\GP(\varphi)\not=\GP(\delta_z)$. Moreover, there is a norm-one homogeneous polynomial $P$ such that the pseudo-hyperbolic distance is attained:
    $$\rho(\varphi,\delta_z)=|\varphi(P)|=1.$$
\end{prop}
\begin{proof}
   Let $r=\|z\|$. Since  $r^p\in\Q\cap (0,1)$  there exist $k,l\in\N$ such that 
   $r^p=\frac{l}{l+k}$. Furthermore, by multiplying $k$ and $l$ by some integer number if necessary, we may assume that $k \ge p$. Note that $s=\left(1-r^p\right)^{1/p}= \left(\frac{k}{l+k}\right)^{1/p}$.
Choose $z' = (z'(1), \ldots, z'(m), 0, 0, \ldots) \in S_{\ell_{p'}} $ satisfying $z'(z)=\|z\|$. 

Given $\varphi\in w^*-ac\{ \delta_{z+s e_{n}} : n> m\}$, since $z+s e_{n}\overset{w}{\to} z$, it is clear that $\varphi\in\M_z$.  To prove our result we will construct a homogeneous polynomial \( P \)  for which $P(z)=0$ and \( P \) attains its norm at \( z + se_n \) for every \( n > m\). To that end, we will consider two polynomials, \( R \) and \( Q \), which attain their norms at \( \frac{z}{\|z\|} \) and \( e_n \), respectively. Using the fact that these polynomials depend on different variables and selecting them with the correct degree, we will conclude that \( RQ \) attains its norm at \( z + se_n \). This is a recurring argument that will be applied, with some technical adjustments, throughout this work.

We define  $R\in\mathcal P(^l\ell_p)$  by $R(x) = (z'(x))^l$ and $Q\in\mathcal P(^k\ell_p)$ by $Q(x) = \sum_{i > m} (e_i'(x))^k$, which clearly have norm one. Let $P\in\mathcal P(^{l+k}\ell_p)$ be given by $P(x)=r^{-l}s^{-k}R(x)Q(x)$. Since $R$ and $Q$ depend on different variables,  Lemma \ref{lem:Pinasco} tells us  that $$\|P\|=\frac{1}{r^ls^k} \|R\|\|Q\| \left(\frac{l^l k^k}{(l+k)^{l+k}}\right)^{\frac1p}=1.$$ Now, for $n> m$,
$$\delta_{z + se_{n}} (P) = \frac{1}{r^l s^k} R (z)  Q (se_{n}) = \frac{1}{r^l s^k} \| z \|^l  s^k=1, $$
which implies that $\varphi(P)=1$.
On the other hand,   $\delta_z(P)=P(z)=0$ indicating that $\rho(\varphi,\delta_z)=|\varphi(P)|=1$. Hence,  $\GP(\varphi)\not=\GP(\delta_z)$. 
\end{proof}

We now work on the arguments from the previous proof to obtain the result for general $z\in B_{\ell_p}$. 

\begin{prop} \label{prop:ele-p-no-racional}
    Let $1<p<\infty$, $z\in B_{\ell_p}$ and $s=\left(1-\|z\|^p\right)^{1/p}$. Then, for every $\varphi\in w^*-ac\{ \delta_{\tc{z}{n}+s e_{n+1}} : n\in \mathbb{N}\}$, we have $\GP(\varphi)\not=\GP(\delta_z)$.
\end{prop}

\begin{proof}
Recalling that $\GP(\delta_0)=\GP(\delta_z)$, to prove our result it suffices to construct a sequence of homogeneous polynomial \( P_j \) (thus satisfying $\delta_0(P_j)=P_j(0)=0$) such that $\varphi(P_j)\underset{j\to\infty}{\longrightarrow} 1 $. 

     Let $r=\|z\|<1$. By Dirichlet's approximation theorem,  there exist sequences $(k_j), (l_j)\subset\N$ such that  
 $\frac{l_j}{l_j+k_j}\underset{j\to\infty}{\rightarrow}r^p$ and $(l_j+k_j)r^p-l_j\underset{j\to\infty}{\rightarrow} 0$.
 Also, 
 after a suitable multiplication if necessary (as in the proof of Proposition \ref{prop:distinta-gleason}),  we may suppose
$k_j\ge p$ for all $j$. 

For  $s=(1-r^p)^{1/p}$ we observe that  $\frac{k_j}{l_j+k_j}\underset{j\to\infty}{\rightarrow}s^p$ and also   $(l_j+k_j)s^p-k_j\underset{j\to\infty}{\rightarrow} 0$. Define
$$
r_j=\left(\frac{l_j}{l_j+k_j}\right)^{\frac1p}\quad \text{ and } \quad s_j=\left(\frac{k_j}{l_j+k_j}\right)^{\frac1p}.
$$ Note that if $r^p\in\Q$ we can take $r_j=r$ and $s_j=s$. If $r^p\not\in\Q$ then $l_j, k_j \rightarrow \infty$ and   it follows from  Dirichlet's conditions that
\begin{equation}\label{eqn:r-s-lim-1}
 \left(\frac{r}{r_j}\right)^{l_j} \underset{j\to\infty}{\longrightarrow} 1 \quad \text{ and } \quad  \left(\frac{s}{s_j}\right)^{k_j} \underset{j\to\infty}{\longrightarrow} 1.   
\end{equation}
For each $m\in\N$, let $\tc{z}{m} = (z(1), z(2), \ldots, z(m), 0, 0, \ldots)\in B_{\ell_p}$ be the $m$-truncation of $z$. The fact that $\|\tc{z}{m}\|\to\|z\|$ allows us to choose  $(m_j)_j\subset \N$ such that $m_j\ge j$ for all $j$ and
\begin{equation}\label{eqn:norm-lim-1}
\left(\frac{\|\tc{z}{m_j}\|}{\|z\|}\right)^{l_j} \underset{j\to\infty}{\longrightarrow} 1.
\end{equation}
Finally, for each  $j$ we take  $z'_{m_j} = (z'_{m_j}(1), \ldots, z'_{m_j}({m_j}), 0, 0, \ldots) \in S_{\ell_{p'}} $ satisfying $z'_{m_j}(\tc{z}{m_j})=\|\tc{z}{m_j}\|$.

Since $\|\tc{z}{n}+se_{n+1}\|^p=\|\tc{z}{n}\|^p + s^p\le 1$ for all $n$ and  $\tc{z}{n}+s e_{n+1}\overset{w}{\to} z$, for any $\varphi\in w^*-ac\{ \delta_{\tc{z}{n}+s e_{n+1}} : n\in \mathbb{N}\}$, we have  $\varphi\in\M_z$.

 Define  norm-one polynomials $R_j\in\mathcal P(^{l_j}\ell_p)$  by $R_j(x) = (z'_{m_j}(x))^{l_j}$ and $Q_j\in\mathcal P(^{k_j}\ell_p)$ by $Q_j(x) = \sum_{i > m_j} (e_i'(x))^{k_j}$, for each $j\in\N$. Consider $P_j\in\mathcal P(^{l_j+k_j}\ell_p)$ given by  $P_j(x)={r_j^{-l_j} s_j^{-k_j}}R_j(x)Q_j(x)$. As in the previous proposition,  from the fact that $R_j$ and $Q_j$ depend on different variables we can apply Lemma \ref{lem:Pinasco} to get $\|P_j\|=\|R_j\|\|Q_j\|=1$.

Now, for $n>m_j$  we have $ \delta_{\tc{z}{n}+s e_{n+1}}(R_j)= z_{m_j}'(\tc{z}{n})^{l_j}=z_{m_j}'(\tc{z}{m_j})^{l_j}=\|\tc{z}{m_j}\|^{l_j} $ and $ \delta_{\tc{z}{n}+s e_{n+1}}(Q_j)= \sum_{i=m_j+1}^n z(i)^{k_j} + s^{k_j}$. Thus, 
\begin{equation}\label{eqn:phi-evaluadaen-Pj}
\varphi(P_j)= \frac{\|\tc{z}{m_j}\|^{l_j}}{r_j^{l_j} s_j^{k_j}}  \Big(\sum_{i>m_j} z(i)^{k_j} + s^{k_j}\Big) = \left(\frac{\|\tc{z}{m_j}\|}{\|z\|}\right)^{l_j}  \left(\frac{r}{r_j}\right)^{l_j} \,\,  \frac{\sum_{i>m_j} z(i)^{k_j} + s^{k_j}}{s_j^{k_j}}.
\end{equation}
Noting that $\sum_{i>m_j} z(i)^{k_j}\to 0$ (since $k_j\geq p$), by \eqref{eqn:r-s-lim-1} and \eqref{eqn:norm-lim-1} we obtain that $\varphi(P_j)\underset{j\to\infty}{\longrightarrow} 1 $. In consequence,  $\GP(\varphi)\not=\GP(\delta_0)=\GP(\delta_z)$.
\end{proof}

Once we know that, for each $z\in B_{\ell_p}$, there is a homomorphism $\varphi\in\M_z$ in a different Gleason part than $\delta_z$, we prove that there are, in fact, many homomorphisms satisfying these conditions. To demonstrate this, we begin by showing that any two distinct elements constructed by the same procedure given above do not lie in the same Gleason part.

\begin{lem}\label{lem:distintos-gleason}
  Let $1<p<\infty$, $z\in B_{\ell_p}$ and $s=\left(1-\|z\|^p\right)^{1/p}$. Given $\varphi_1, \varphi_2 \in  w^*-ac\{ \delta_{\tc{z}{n}+s e_{n+1}} : n\in \mathbb{N}\}$, we have $\GP(\varphi_1)=\GP(\varphi_2) $ if and only if $\varphi_1=\varphi_2.$ 
\end{lem}
\begin{proof}
    If $\varphi_1\neq \varphi_2$, there are disjoint countably sets $\mathbb{N}_1$ and $\mathbb{N}_2$ such that $\mathbb{N}_1\cup \mathbb{N}_2=\mathbb N$ and $\varphi_u \in w^*-ac\{ \delta_{\tc{z}{n}+s e_{n+1}} : n+1\in \mathbb{N}_u\}$ for $u=1,2$. 

    Take $m_j$, $k_j$, $l_j$ and $z'_{m_j}$ as in Proposition \ref{prop:ele-p-no-racional} and define $R_j\in\mathcal P(^{l_j}\ell_p)$  by $R_j (x)= (z'_{m_j}(x))^{l_j}$ and $Q_j^1\in\mathcal P(^{k_j}\ell_p)$ by $Q_j^1(x) = \displaystyle{\sum_{\substack{i > m_j \\ i\in \N_1}}} (e_i'(x))^{k_j}$. Note that these polynomials are essentially the same as those in Proposition \ref{prop:ele-p-no-racional}, with the difference that  $Q_j^1$ involves only coordinates belonging to $\N_1$. 

Proceeding like in Proposition \ref{prop:ele-p-no-racional}  (see \eqref{eqn:phi-evaluadaen-Pj}), we can see  that the sequence $\varphi_1(r_j^{-l_j}s_j^{-k_j}{ R_j Q_j^1})$ converges to $1$. Now, let us show  that $\varphi_2(r_j^{-l_j}s_j^{-k_j}{ R_j Q_j^1})$ converges to $0$, which gives the desired conclusion. Indeed, since the polynomial $Q_j^1$ does not depend on the coordinates on $\N_2$,   
\begin{eqnarray*}
\varphi_2(r_j^{-l_j}s_j^{-k_j}{ R_j Q_j^1}) &=& \left(\frac{\|\tc{z}{m_j}\|}{\|z\|}\right)^{l_j}  \left(\frac{r}{r_j}\right)^{l_j}  \frac{Q_j^1(z)}{s_j^{k_j}} \\ &=&  \left(\frac{\|\tc{z}{m_j}\|}{\|z\|}\right)^{l_j}  \left(\frac{r}{r_j}\right)^{l_j}  \frac{\sum_{\substack{i > m_j \\ i\in \N_1}} z(i)^{k_j} }{s_j^{k_j}}\rightarrow 0, 
\end{eqnarray*}
because, with the chosen $m_j$, $ k_j$ and $l_j$, we have $\left(\frac{\|\tc{z}{m_j}\|}{\|z\|}\right)^{l_j} \to 1$, $\left(\frac{r}{r_j}\right)^{l_j} \to 1$ and then 
$$\left|\frac{\sum_{\substack{i > m_j \\ i\in \N_1}} z(i)^{k_j} }{s_j^{k_j}}\right|\le \left(\frac{\|\tcm{z}{m_j}\|}{s_j}\right)^{k_j}\to 0. \qedhere$$
\end{proof}

By refining the construction in the previous lemma, we arrive at the promised answer to the question posed in \cite{DimLasMae}. Since we know that any two elements in $w^*-ac\{ \delta_{\tc{z}{n}+s e_{n+1}} : n\in \mathbb{N}\}$ belong to different Gleason parts, it remains to show that this subset of $\M_z$ has cardinality $2^{\mathfrak c}$.  The next proposition is the final step to prove that $\M_z$ intersects $2^{\mathfrak c}$ different Gleason parts.

\begin{prop}\label{prop:resultado-ppal-ellp}
   Let $1<p<\infty$, $z\in B_{\ell_p}$ and $s=\left(1-\|z\|^p\right)^{1/p}$. Then, the cardinality of
$w^*-ac\{ \delta_{\tc{z}{n}+s e_{n+1}} : n\in \mathbb{N}\}\subset \M_z$
is $2^{\mathfrak c}$.
\end{prop}

\begin{proof}

Let us write the natural numbers as a disjoint union of infinite sets
$\N = \bigcup_{u=1}^\infty \N_u.$
 For each $u\in\N$ take
$$\varphi^u\in w^*-ac\{ \delta_{\tc{z}{n}+s e_{n+1}} : n+1\in \mathbb{N}_u\}.$$
Our goal is to show that the sequence $\{\varphi^u\}$ is an interpolating sequence. For this, given $b\in \ell_\infty$ we will  construct a polynomial $S$ such that $\varphi^u(S)=b(u)$. 
 Take $m_j$, $k_j$, $l_j$ , $z'_{m_j}$ and $R_j\in\mathcal P(^{l_j}\ell_p)$ as in Lemma \ref{lem:distintos-gleason},  and for each $u$ define $Q_j^u\in\mathcal P(^{k_j}\ell_p)$ similarly as $Q_j^1$ was defined: $$Q_j^u(x) = \displaystyle{\sum_{\substack{i > m_j \\ i\in \N_u}}} (e_i'(x))^{k_j}.$$
We also define a \emph{translation} of $Q_j^u$:
$$V_j^u(x):=Q_j^u(x-z).$$
Since $Q_j^u$ does not depend on the first $m_j$ coordinates,  we have 
$$V_j^u(x)=Q_j^u(x- \tcm{z}{m_j})$$ 
and then, for any $x\in B_{\ell_p}$,
\begin{eqnarray*}
    |V_j^u(x)-Q_j^u(x)|&=& |Q_j^u(x-\tcm{z}{m_j})- Q_j^u(x)| \\
    &\leq & 2^{k_j}L(k_j) \| \tcm{z}{m_j}\|,\
\end{eqnarray*}
 where $L(k_j)$ is the Lipschitz constant from \eqref{eqn:Lipschitz}. 
At this point note that we may assume that $m_j$ grows fast enough to satisfy
\begin{equation}\label{eqn:condicion-extra-mj}
    \frac{\|V_j^u-Q_j^u\|}{s_j^{k_j}}\leq \frac{2^{k_j}L(k_j)}{s_j^{k_j}}\| \tcm{z}{m_j}\|\underset{j\to\infty}{\longrightarrow} 0.
\end{equation}
As before, we have $\varphi^u(r_j^{-l_j}s_j^{-k_j}{ R_j Q_j^u})$ converges to $1$.  Let us see that the convergence is uniform on $u$. Indeed,
\begin{eqnarray}
  \nonumber 1 &\ge& \Bigg|\varphi^u\left(\frac{ R_j Q_j^u}{r_j^{l_j}s_j^{k_j}}\right)\Bigg| = \left(\frac{\|\tc{z}{m_j}\|}{\|z\|}\right)^{l_j}  \left(\frac{r}{r_j}\right)^{l_j}  \Bigg| \frac{s^{k_j}+ \sum_{ \substack{i > m_j \\ i\in \N_u}} z(i)^{k_j} }{s_j^{k_j}}\Bigg|\\
   &\geq & \left(\frac{\|\tc{z}{m_j}\|}{\|z\|}\right)^{l_j}  \left(\frac{r}{r_j}\right)^{l_j}  \left(  \frac{s^{k_j} -\sum_{i>m_j} |z(i)|^{k_j} }{s_j^{k_j}}\right)  \underset{j\to\infty}{\longrightarrow} 1.  \label{eq:uniforme}
\end{eqnarray}

We now define $P^u_j(x)={r_j^{-l_j}s_j^{-k_j}} { R_j(x) V_j^u(x)}$  and show that the numbers $t^u_j:=\varphi^u\left(P^u_j\right)$ are uniformly away from $0$ for $j$ large enough:
\begin{eqnarray*}
    \left| \varphi^u\left(P^u_j\right) \right| &\geq& \Bigg|\varphi^u\left(\frac{ R_j Q_j^u}{r_j^{l_j}s_j^{k_j}}\right)\Bigg| - \Bigg|\varphi^u\left(\frac{ R_j }{r_j^{l_j}s_j^{k_j} }(V_j^u-Q_j^u) \right)\Bigg| \\
    &\ge &  \Bigg|\varphi^u\left(\frac{ R_j Q_j^u}{r_j^{l_j}s_j^{k_j}}\right)\Bigg| - \left(\frac{\|\tc{z}{m_j}\|}{\|z\|}\right)^{l_j}  \left(\frac{r}{r_j}\right)^{l_j} \frac{2^{k_j}L(k_j)}{s_j^{k_j}}\| \tcm{z}{m_j}\|.
\end{eqnarray*}
The last expression converges to $1$
uniformly on $u$ by  \eqref{eqn:condicion-extra-mj} and \eqref{eq:uniforme}.  
Then, we can fix  $j$ such that 
$$|t^u_j|=\left|\varphi^u\left(P^u_j\right)\right|\geq \frac{1}{2}$$
for all $u$ and define, for $x\in \ell_p$,
$$ S(x)= \sum_{u=1}^\infty  \frac{b(u)}{t^u_j} P^u_j(x)=\frac 1 {r_j^{l_j}s_j^{k_j}} { R_j(x) }   \sum_{u=1}^\infty  \frac{b(u)}{t^u_j}  V_j^u(x).$$  Given that the polynomials  $Q^u_j$ have disjoint support and degree greater than or equal to $p$, it is easy to see that  
$\sum_{u=1}^\infty  \frac{b(u)}{t^u_j}  Q_j^u(x)$
is continuous with norm not greater than $2\|b\|_\infty$. Since $V_j^u$ is a translation of $Q_j^u$, the polynomial  $S$ is well defined and continuous.

We need to see that  $\varphi^u(S)=b(u)$ for each $u$. To avoid further complications in an already cumbersome notation, we do it for $u=1$. By definition of $t_j^1$, we have 
$$\varphi^1\left(\frac{b(1)}{t_j^1} P_j^1\right) =t_j^1 \frac{b(1)}{t_j^1} =b(1).$$
It remains to show that
$$\varphi^1 \left(\sum_{u>1}^\infty \frac{b(u)}{t^u_j} P^u_j\right)=0.$$ Note that the series does not converge in norm, so we cannot simply swap $\varphi^1$ and the sum. 
Recalling that $\varphi^1\in w^*-ac\{ \delta_{\tc{z}{n}+se_{n+1}} : n+1\in \mathbb{N}_1\}$ and that for $u>1$ the polynomials $Q^u_j$ do not depend on the coordinates in $\mathbb{N}_1$, the result follows from the computation
\begin{eqnarray*}
    \Bigg|\delta_{\tc{z}{n}+se_{n+1}}\left(\sum_{u>1}^\infty \frac{b(u)}{t^u_j}V_j^u\right)\Bigg| &=&  \Bigg|\sum_{u>1}^\infty \frac{b(u)}{t^u_j} V_j^u(\tc{z}{n}+s e_{n+1})\Bigg| \\
    &=& \Bigg|\sum_{u>1}^\infty \frac{b(u)}{t^u_j}Q_j^u(\tc{z}{n}-z+s e_{n+1})\Bigg|\\
    &=& \Bigg|\sum_{u>1}^\infty \frac{b(u)}{t^u_j}\displaystyle{\sum_{\substack{i > m_j \\ i\in \N_u\\ i>n}}} (-z(i))^{k_j}\Bigg|\leq 2\|b\|_\infty\sum_{i>n} |z(i)|^{k_j} \underset{n\to\infty}{\longrightarrow} 0.\
\end{eqnarray*}
Then, 
\begin{equation*}
\varphi^1 \left(\sum_{u>1}^\infty \frac{b(u)}{t^u_j} P^u_j\right)=\varphi^1 \left(\sum_{u>1}^\infty \frac{b(u)}{t^u_j} \frac{ R_j V_j^u}{r_j^{l_j}s_j^{k_j}} \right)=\frac{ \varphi^1(R_j)}{r_j^{l_j}s_j^{k_j}} \varphi^1 \left( \sum_{u>1}^\infty \frac{b(u)}{t^u_j} V_j^u \right)=0.
\end{equation*}

Thus,  $\{\varphi^u\}$ is an interpolating sequence and then its closure is homeomorphic to $\beta(\N)$, the Stone–\v{C}ech compactification of the natural numbers. In particular, its cardinality is $2^\mathfrak{c}$.
\end{proof}

It remains to prove the final assertion in Theorem \ref{thm:resultado-ppal-ellp}: namely, that all the Gleason parts constructed in Lemma \ref{lem:distintos-gleason} and Proposition \ref{prop:resultado-ppal-ellp} are contained in the fiber over $z$. First, we need the following.

\begin{lem}\label{lem:existe_m}
    Let $1< p<\infty$. Given $\omega\in B_{\ell_p}\setminus\{0\}$ and $\eta >0$ there is $m=m(\eta,\omega)$ such that for any $\psi \in \M_\omega$, $l\in \N$, $k \geq p$, and any norm-one $k$-homogeneous polynomial $Q$ that does not depend on the first $m$ variables, we have
    $$|\psi(Q) |\leq \left(\frac{l^lk^k}{(l+k)^{l+k}}\right)^{\frac{1}{p}} \frac{(1+\eta)^l}{\|\omega\|^l}.$$
\end{lem}
\begin{proof}
    As before we write $\tc{\omega}{m} = (\omega(1), \omega(2), \ldots, \omega(m), 0, 0, \ldots)\in B_{\ell_p}$ and take  $\omega'_{m} = (\omega'_{m}(1), \ldots, \omega'_{m}({m}), 0, 0, \ldots) \in S_{\ell_{p'}} $ satisfying $\omega'_{m}(\tc{\omega}{m})=\|\tc{\omega}{m}\|$. We choose $m$  large enough so that  $\|\tc{\omega}{m}\|\geq {\|\omega\|}/{(1+\eta)}$ and consider the norm-one polynomial $R\in\mathcal P(^l\ell_p)$ given by $R(x)=(\omega'_{m}(x))^l$. For $\psi \in \M_\omega$ we have 
    $$\psi (R)=R(\omega)= \|\tc{\omega}{m}\|^l.$$

    Let  $P(x)=\left(\frac{(l+k)^{l+k}} {l^lk^k}\right)^{1/p}R(x)Q(x)\in\mathcal P(^{l+k}\ell_p)$. Appealing once again to Lemma \ref{lem:Pinasco}, we conclude that $\|P\|=1$. Then,
    $$
          \left(\frac{(l+k)^{l+k}} {l^lk^k}\right)^{\frac{1}{p}} \|\tc{\omega}{m}\|^l   |\psi(Q)|= |\psi(P)|\leq 1.
    $$
    Therefore,
    $$|\psi(Q)| \leq  \left(\frac{l^lk^k}{(l+k)^{l+k}} \right)^{\frac{1}{p}}\frac{1}{\|\tc{\omega}{m}\|^l}\leq \left(\frac{l^lk^k}{(l+k)^{l+k}} \right)^{\frac{1}{p}} \frac{(1+\eta)^l}{\|\omega\|^l}. \qedhere$$
\end{proof}
\begin{rem}\label{rem:unif-convex}
    From \cite[Thm. 2.1]{kim2014uniform}, if $X$ is a uniformly convex Banach space, for any $\varepsilon>0$ there is $\eta(\varepsilon)>0$ satisfying the following:   
 \begin{quotation}
        For $z_0, \omega\in S_X$ and $z_0'\in S_{X'}$ with $z_0'(z_0)=1$ such that $\|z_0 - \lambda \omega\| > \varepsilon $, for all $  |\lambda|=1$
      we have $|z_0'(\omega)| < 1-\eta$.
    \end{quotation}

\end{rem}

\begin{prop}\label{prop:gleason-no-intersec}
 Let $1<p<\infty$, $z\in B_{\ell_p}$ and $s=\left(1-\|z\|^p\right)^{1/p}$. Then, $\GP(\varphi)\subset\M_z$ for every $\varphi\in w^*-ac\{ \delta_{\tc{z}{n}+s e_{n+1}} : n\in \mathbb{N}\}$.
\end{prop}

\begin{proof}
    Let us see that $\psi \notin \GP(\varphi)$ for any $\psi \in \M_\omega$ with $z \neq \omega\in B_{\ell_p}$ (note that from Remark \ref{rmk: PG-fibras} this is obvious when $\|\omega\|=1$).  
    We split the proof into three cases. 
   First, we deal with the simplest one:  $z=0\neq \omega$. Note that here  $\varphi$ is a $w^*$-limit point of  $\delta_{e_n}$.
    We define $V_j(x) = \sum_{i>j} (e_i'(x))^k$ for some $k>p$ such that 
$$\left(\frac{k^k}{(k+1)^{k+1}}\right)^{\frac{1}{p}} \frac{2}{\|\omega\|}<1.$$ Then, for any $j$, $V_j$ is a norm-one polynomial satisfying $\varphi(V_j)= 1$.  But  Lemma \ref{lem:existe_m} (with $l=1$ and $\eta=1$) gives us that, if $j$ is large enough, then
$$|\psi(V_j)|\leq \left(\frac{k^k}{(k+1)^{k+1}}\right)^{\frac{1}{p}} \frac{2}{\|\omega\|} <1. $$ Then, $\psi(V_j)$ cannot converge to $1$, concluding that $\psi\not\in \GP(\varphi)$.

For the remaining steps, we use the polynomials $R_j$, $Q_j$, and $P_j$ introduced in the proof of Proposition \ref{prop:ele-p-no-racional}, so we recall their definitions. Denoting $r=\|z\|<1$, there are sequences $(k_j), (l_j)\subset\N$ satisfying 
 $\frac{l_j}{l_j+k_j}\underset{j\to\infty}{\rightarrow}r^p$, $(l_j+k_j)r^p-l_j\underset{j\to\infty}{\rightarrow} 0$ and
$k_j\ge p$ for all $j$. We set $R_j(x) = (z'_{m_j}(x))^{l_j}$ and $Q_j(x) = \sum_{i > m_j} (e_i'(x))^{k_j}$, whose normalized product  is the polynomial $P_j\in\mathcal P(^{l_j+k_j}\ell_p)$ given by  $$P_j(x)=\left(\frac{(l_j+k_j)^{l_j+k_j}}{l_j^{l_j}  k_j^{k_j}} 
 \right)^{\frac1p}R_j(x) Q_j(x).$$
 
We now consider the second case: $z\neq 0$ and $\omega = \lambda z$ for some $\lambda \in \mathbb{C}$. For $\psi \in \mathcal{GP}(\varphi)$ we show that if $\psi \in \mathcal{M}_\omega$, then necessarily $\lambda = 1$. Since $\{P_j\}$ are norm-one polynomials satisfying $\varphi(P_j) \rightarrow 1$ (check the proof of Proposition \ref{prop:ele-p-no-racional} for details), it follows that $\psi(P_j) \rightarrow 1$ as well. Now, we take, for each $j$, the  norm-one polynomial $S_j\in\mathcal P(^{l_j+k_j+1}\ell_p)$ given by 
\[
S_j(x) = \left( \frac{(l_j+k_j+1)^{l_j+k_j+1}}{(l_j+1)^{l_j+1}  k_j^{k_j}}\right)^{\frac1p}z_{m_j}'(x) R_j(x) Q_j(x).
\]
Note that $S_j$ is essentially  $z_{m_j}' P_j$, with the constant adjusted to have  unit norm. We will show that $\varphi(S_j) \rightarrow 1$, while $\psi(S_j) \rightarrow \lambda$, which implies that $\lambda = 1$, as desired. 

At this point, it is important to observe that if $r^p \notin \Q$, then $l_j, k_j \to \infty$. For $r^p=\frac{l}{l+k} \in \Q$, we can set $l_j = j \cdot l$ and $k_j = j \cdot k$ to also get $l_j, k_j \to \infty$. This observation allows us to proceed regardless of $r^p$ being rational or not.

Note that
\[
\varphi(S_j) = \left( \frac{(l_j+k_j+1)^{l_j+k_j}}{(l_j+k_j)^{l_j+k_j}} \frac{l_j^{l_j}}{(l_j+1)^{l_j}} \frac{l_j+k_j+1}{l_j+1} \right)^{\frac{1}{p}}\varphi(z_{m_j}') \varphi(P_j).
\]
Since $\varphi$ is in the fiber of $z$,  $\varphi(z_{m_j}')=z_{m_j}'(z) = \| \tc{z}{m_j} \|$, which converges to $\| z \| = r$. We already know that $\varphi(P_j)$ converges to $1$. Using the fact that $l_j,k_j\rightarrow \infty$ the first term inside the parenthesis converges to $e$ and the second to $e^{-1}$. Since $\frac{l_j}{l_j+k_j} \rightarrow r^p$, the last term inside the parenthesis converges to $r^{-p}$. All these imply that $\varphi(S_j) \rightarrow 1$. 

Similarly, for $\psi$, we have
\[
\psi(S_j) = \left( \frac{(l_j+k_j+1)^{l_j+k_j}}{(l_j+k_j)^{l_j+k_j}} \frac{l_j^{l_j}}{(l_j+1)^{l_j}} \frac{l_j+k_j+1}{l_j+1} \right)^{\frac{1}{p}}\psi(z_{m_j}') \psi(P_j).
\]
Given that $\psi$ is in the fiber of $\lambda z$, $\psi(z_{m_j}')=z_{m_j}'(\lambda z) = \lambda \| \tc{z}{m_j} \|$, which converges to $\lambda \| z \| = \lambda r$. Our assumption that $\psi \in \mathcal{GP}(\varphi)$ yields that $\psi(P_j)$ also converges to $1$. The parenthesis is the same as before and converges to $r^{-p}$, completing the proof for this case.

  Finally, we deal with the last case: $\omega$ and $z$ are linearly independent.   We know that    $\varphi(P_j)\rightarrow 1,$
    so we only need to show that $\psi(P_j)$ is away from $1$ for $j$ large enough.
Given that $\psi\in \M_\omega$, we know that $\psi(R_j)=R_j(\omega)=z'_{m_j}(\omega)^{l_j}$. Since $z$ and $\omega$ are linearly independent, there exists $\varepsilon > 0$ such that 
 for every $\lambda\in \C$ with $|\lambda|=1$,  
\begin{equation*}
    \left\| \frac{z}{\|z\|} - \lambda \frac{\omega}{\|\omega\|} \right\| > 2\varepsilon.
\end{equation*}
    
    Given that $\tc{z}{m_j} \rightarrow z$, there is $j_0$ such that  for any $j>j_0$ we have

\begin{equation}\label{eqn:uniform conv}
    \left\| \frac{\tc{z}{m_j}}{\|\tc{z}{m_j}\|} - \lambda \frac{\omega}{\|\omega\|} \right\| > \varepsilon,
\end{equation}
for all $\lambda\in \C$ with $|\lambda|=1$.

Using that $z'_{m_j}$ is the unique element in $S_{\ell_{p'}}$ satisfying $z'_{m_j}(\tc{z}{m_j}) = \|\tc{z}{m_j}\|$, by Remark \ref{rem:unif-convex} and  \eqref{eqn:uniform conv},  there exists $\eta=\eta(\varepsilon) > 0$ such that 
\begin{equation}\label{eqn:lejos de w}
    |z'_{m_j}(\omega)| \leq \|\omega\| (1 - \eta)\,\,\,\forall j > j_0.
\end{equation}

    By Lemma \ref{lem:existe_m}, eventually replacing $j_0$ by a larger number, we may assume that 
\begin{equation}\label{eqn:aplic del lema}
    |\psi(Q_j)| \leq \left(\frac{l_j^{l_j}k_j^{k_j}}{(l_j+k_j)^{l_j+k_j}}\right)^{\frac{1}{p}}  \frac{(1+\eta)^{l_j}}{\|\omega\|^{l_j}}\,\,\,\text{for all } j > j_0.
\end{equation}
Combining \eqref{eqn:lejos de w} and \eqref{eqn:aplic del lema} we obtain the desired conclusion: if $j>j_0$,
\begin{eqnarray*}
|\psi(P_j)|&=& \Bigg|\left(\frac{(l_j+k_j)^{l_j+k_j}}{l_j^{l_j}k_j^{k_j}}\right)^{\frac{1}{p}}  \psi(R_j) \psi(Q_j) \Bigg| 
         \leq \frac{(1+\eta)^{l_j}}{\|\omega\|^{l_j}} \|\omega\|^{l_j} (1-\eta)^{l_j}\\
         &=&  (1+\eta)^{l_j}(1-\eta)^{l_j} 
         =(1-\eta^2)^{l_j} \leq 1-\eta^2. \qedhere
\end{eqnarray*}
\end{proof}

Finally, we show that each $\varphi \in w^*-ac\{\delta_{\tc{z}{n} + s e_{n+1}} : n \in \mathbb{N}\}$ is the only element in $\GP(\varphi)\cap \overline{\{\delta_x: x\in B_{\ell_p}\}}^{w^*}$ (i.e., $\varphi$ is the only element in $\GP(\varphi)$ outside the corona).
 The sketch of the argument is the following: suppose that $\psi = w^*-\lim_\beta \delta_{z + x_\beta}$ belongs to $\GP(\varphi)$. Then, for sufficiently large $\beta$, the polynomials of the form $\sum_{i > m_j} e_i'(x)^k$ should \emph{almost} attain their norms at $\frac{x_\beta}{\|x_\beta\|}$. This, in turn, forces $x_\beta$ to be close to a multiple of a canonical vector, giving that $\psi \in w^*-ac \{\delta_{\tc{z}{n} + s e_{n+1}} : n \in \mathbb{N}\}$. Consequently, Lemma \ref{lem:distintos-gleason} yields $\psi = \varphi$. To ensure clarity, the proof is divided into several parts, as it involves various technical details.

We begin with the following lemmas that give us information about the norming sets for polynomials of the form $\sum_{i \in I} e_i'(x)^k=\sum_{i \in I} x(i)^k$.

\begin{lem} \label{lem: z casi canonico}
    Let $1<p<\infty$. Given $\varepsilon>0$ there is $\delta=\delta(p,\varepsilon)$  such that, if $z\in \overline{B}_{\ell_p}$ satisfies \
    $$\sum_{i\in \N} |z(i)|^k>1-\delta$$ for some $k>p$, then $\|z - e_\mathfrak{m} \|< \varepsilon$, where $\mathfrak m$ satisfies $|z(\mathfrak{m})|=\max_{i\in \N}|z(i)|$. 
\end{lem}

\begin{proof}
    Let $k_0 = [p]+1$ be the smallest integer strictly larger than $p$.  Take $\delta>0$ such that, if $(1-\delta) <x^{k_0-p}\leq 1$, then 
\begin{equation}\label{eqn:delta}
        |1-x|^p+1-x^p< \varepsilon^p.
    \end{equation}
    Next, suppose that
\begin{eqnarray*}
  1-\delta &<&  \sum_{i\in \N} |z(i)|^k=\sum_{i\in \N} |z(i)|^p|z(i)|^{k-p} \\
  &\leq& \max_{i\in \N}|z(i)|^{k-p} \sum_{i\in \N} |z(i)|^p \le\max_{i\in \N}|z(i)|^{k-p}.  \
\end{eqnarray*}
    Then, taking  $\mathfrak{m}$ such that $|z(\mathfrak{m})|=\max_{i\in \N}|z(i)|$ we  have $1-\delta<|z(\mathfrak{m})|^{k-p}\leq |z(\mathfrak{m})|^{k_0-p}$.  Combining this with \eqref{eqn:delta} we get
    $$\|z-e_\mathfrak{m} \|^p= |z(\mathfrak{m})-1|^p + \sum_{i\neq \mathfrak{m}} |z(i)|^p\le |1-z(\mathfrak{m})|^p + 1-|z(\mathfrak{m})|^p <\varepsilon^p,$$ 
    as we wanted to see.
\end{proof}
Note that in the previous lemma, the choice of $\mathfrak{m}$ is not necessarily unique in general. However, for small values of $\varepsilon$ (namely, for  $\varepsilon< 2^{-1+1/p}$),  any  $z$ satisfying the hypotheses must attain its maximum at a single coordinate. 

\begin{lem}\label{lem:cambiar por canonicos}
    Let $\{x_\beta\}_\beta\subset \overline{B}_{\ell_p}$ be a net such  
    $$\lim_{j\to\infty} \lim_\beta \sum_{i\in I_j} x_\beta(i)^{k_j} =1,$$
   where $\{k_j\}$ is a sequence of integers with $k_j>p$ and $I_j$ are subsets of $\N$. Then, if we take $\mathfrak{m}_\beta$ for which $|x_\beta(\mathfrak{m}_\beta)|=\max|x_\beta(i)|$, we have 
    $$\lim_\beta\| x_\beta - e_{\mathfrak{m}_\beta} \|= 0.$$
\end{lem}
\begin{proof}
    Fix $\varepsilon>0$ and take $\delta=\delta(p,\varepsilon)$ given by the previous lemma.  We know that there is $j$ such that  
    $$\lim_\beta \Bigg|\sum_{i\in I_j} x_\beta(i)^{k_j} \Bigg| > 1-\delta.$$
    Now take $\beta_0$ such that if $\beta \geq \beta_0$, then 
    $$\sum_{i\in I_j}|x_\beta(i)|^{k_j}\ge \big|\sum_{i\in I_j} x_\beta(i)^{k_j} \big| >1-\delta.$$
    By the previous lemma, we conclude that $\| x_\beta - e_{\mathfrak{m}_\beta} \|<\varepsilon$ for all such $\beta$, which finishes the proof.
\end{proof}

Now, we have all the elements needed to address the promised result.

\begin{prop}\label{prop:nocorona} Let $1<p<\infty$, $z\in B_{\ell_p}$ and $s=\left(1-\|z\|^p\right)^{1/p}$. Then, if  $\varphi\in w^*-ac\{ \delta_{\tc{z}{n}+s e_{n+1}} : n\in \mathbb{N}\}$, then $$\GP(\varphi)\cap \overline{\{\delta_x: x\in B_{\ell_p}\}}^{w^*}=\{\varphi\}.$$ Consequently, if the corona 
is empty,  $\GP(\varphi)$ reduces to the singleton $\{\varphi\}$.
\end{prop}
\begin{proof}
    Take $\varphi \in w^*-ac\{ \delta_{\tc{z}{n}+s e_{n+1}} : n\in \mathbb{N}\}$ and $\psi \in \M_z\cap \overline{\{\delta_x: x\in B_{\ell_p}\}}^{w^*}$. By the comments following Lemma \ref{lem:redes_cerca}, there is a net $\{x_\beta\}_\beta \subset S_{\ell_p}$ such that $\psi= w^*-\lim_\beta  \delta_{x_\beta}$. Let us assume that $\psi\in\GP(\varphi)$ and show that, then, $\psi=\varphi$. 
  
    We aim to prove that $\psi$ actually belongs to $w^*-ac\{ \delta_{\tc{z}{n}+s e_{n+1}} : n\in \mathbb{N}\}$, so that   Lemma \ref{lem:distintos-gleason} gives the desired result.     
    The strategy is to use Lemma \ref{lem:redes_cerca} repeatedly to change the net $\{x_\beta\}_\beta$ by a net of the form $\{\tc{z}{n_\beta -1}+s e_{n_\beta}\}_\beta$. 

    \emph{First change:} We begin by changing the net $\{x_\beta\}_\beta$ by a net formed by truncations of $z$ added to vectors of disjoint support.  
    Since $e_1'(x_\beta)\rightarrow e_1'(z)$, we can assume $|z(1)-x_\beta(1)| <1$ for all $\beta$. Also, the condition $x_\beta \neq z$ for every $ \beta$ allows us to define 
    $${n_\beta}=\max\left\{ n : \| \tc{z}{n} - \tc{x_\beta}{n}\| < \frac{1}{n}\right\}.$$
   Let us show that $n_\beta \underset{\beta}{\longrightarrow} \infty$. As $z$ is the weak-limit of  $\{x_\beta\}_\beta$, given $n_0$ we choose $\beta_0$ such that if $\beta \geq\beta_0$, 
    $$|z(i)-x_\beta(i)|=|e_i'(z)-e_i'(x_\beta)| < \frac{1}{n_0^2}, \,\,\,   i=1,\ldots,n_0.$$
    This implies that 
    $$\| \tc{z}{n_0} - \tc{x_\beta}{n_0} \|\leq \sum_{i=1}^{n_0} |z(i)-x_\beta(i)| < \frac{1}{n_0},$$
    so we must have $n_\beta\geq n_0$ for all $\beta \geq \beta_0$, proving the announced limit.

   Now, consider the net $\{\tc{z}{n_\beta}   + \tcm{x_\beta}{n_\beta}\}_\beta$ which satisfies
    $$\|x_\beta  -\left( \tc{z}{n_\beta}   +\tcm{x_\beta}{n_\beta}\right)\| = \|\tc{x_\beta}{n_\beta}-\tc{z}{n_\beta}\| \leq \frac{1}{n_\beta}\underset{\beta}{\longrightarrow} 0.$$

    Since $\|\tc{z}{n_\beta}\|\leq \|z\|=r$  the following net is contained in $\overline{B}_{\ell_p}$:
    \begin{equation}\label{eq:firstchange}
        \left\{\tc{z}{n_\beta}   +\frac{s}{\|\tcm{x_\beta}{n_\beta}\|}\tcm{x_\beta}{n_\beta}\right\}_\beta,
    \end{equation} with the proviso that if $\|\tcm{x_\beta}{n_\beta}\| = 0$, we simply consider $\tc{z}{n_\beta}$.

     Given that $\|\tc{z}{n_\beta}\|\rightarrow \|z\|=r$  and  $\|x_\beta\|=1$ for all $\beta$ we know that $\|\tc{x_\beta}{n_\beta}\|\rightarrow r$ and, then,   $\|\tcm{x_\beta}{n_\beta}\|\rightarrow s$, which gives 
    $$\lim_\beta\left\|x_\beta  -\left(\tc{z}{n_\beta}  + \frac{s}{\|\tcm{x_\beta}{n_\beta}\|}\tcm{x_\beta}{n_\beta}\right)\right\| =\lim_\beta  \|x_\beta  -\left( \tc{z}{n_\beta}   +\tcm{x_\beta}{n_\beta}\right)\|=0.$$
   Thus, we can apply Lemma \ref{lem:redes_cerca} to replace $\{x_\beta\}_\beta$ by the net given in \eqref{eq:firstchange}. Redefining $x_\beta$ as the second term of this net, we now have $$\psi=\lim_\beta  \delta_{\tc{z}{n_\beta}+x_\beta},$$ with $\tc{z}{n_\beta}$ and $x_\beta$ of disjoint support, $\|x_\beta\|= s$ and $n_\beta \underset{\beta}{\longrightarrow} \infty$. 

    \emph{Second change:} 
     In this step, we aim to change the vectors $x_\beta$ by multiples of canonical vectors. Once again, as in the proof of Proposition \ref{prop:ele-p-no-racional}, we take sequences $(k_j), (l_j)\subset\N$ satisfying 
 $\frac{l_j}{l_j+k_j}\underset{j\to\infty}{\rightarrow}r^p$, $(l_j+k_j)r^p-l_j\underset{j\to\infty}{\rightarrow} 0$ and
$k_j > p$ for all $j$ and define the polynomials $R_j(x) = (z'_{m_j}(x))^{l_j}$ and $Q_j(x) = \sum_{i > m_j} (e_i'(x))^{k_j}$ with normalized product  $P_j\in\mathcal P(^{l_j+k_j}\ell_p)$ given by  $$P_j(x)=\left(\frac{(l_j+k_j)^{l_j+k_j}}{l_j^{l_j}  k_j^{k_j}} 
 \right)^{\frac1p}R_j(x) Q_j(x).$$ 
 Since $\varphi(P_j)\rightarrow 1$ (see the proof of Proposition \ref{prop:ele-p-no-racional}), then also $\psi(P_j)\rightarrow 1$ and thus, 
    \begin{eqnarray*}
        1&=& \lim_j \psi(P_j) = \lim_j\lim_\beta P_j(\tc{z}{n_\beta}+x_\beta)  \\
        &=& \lim_j\lim_\beta \frac{R_j(\tc{z}{n_\beta})}{r_j^{l_j}} \left( \frac{Q_j(\tc{z}{n_\beta})}{s_j^{k_j}}+\frac{Q_j(x_\beta)}{s_j^{k_j}}\right) ,\
    \end{eqnarray*} 
    where we used that $R_j(\tc{z}{n_\beta}+x_\beta)=R_j(\tc{z}{n_\beta})$ for $\beta$ large enough.     Arguing as in the proofs of Proposition \ref{prop:ele-p-no-racional} and Lemma \ref{lem:distintos-gleason}, we can see that  $$\lim_j\lim_\beta \frac{R_j(\tc{z}{n_\beta})}{r_j^{l_j}} = \lim_j \frac{\|\tc{z}{m_j}\|^{l_j}}{r^{l_j}}\frac{r^{l_j}}{r_j^{l_j}}=1\qquad \text{and}$$   $$\lim_j\lim_\beta \frac{Q_j(\tc{z}{n_\beta})}{s_j^{k_j}}=\lim_j \frac{\sum_{\substack{i > m_j }} z(i)^{k_j} }{s_j^{k_j}}=0,$$
       which gives 
    $$1=\lim_j\lim_\beta \frac{Q_j\left(x_\beta\right)}{s_j^{k_j}} =\lim_j\lim_\beta Q_j\left(\frac{x_\beta}{s}\right)\frac{s^{k_j}}{s_j^{k_j}} =\lim_j\lim_\beta Q_j\left(\frac{x_\beta}{s}\right).$$
    By Lemma  \ref{lem:cambiar por canonicos}, we have
    $$\| x_\beta - s e_{\mathfrak{m}_\beta} \|\rightarrow 0,$$
    where $\mathfrak{m}_\beta$ satisfies $|x_\beta(\mathfrak{m}_\beta)|=\max|x_\beta(i)|$.
    Then, using again Lemma \ref{lem:redes_cerca}, we may replace $\tc{z}{n_\beta}+x_\beta$ by $\tc{z}{n_\beta}+s e_{\mathfrak{m}_\beta}$. Then, we now have $$\psi=\lim_\beta  \delta_{\tc{z}{n_\beta}+se_{\mathfrak{m}_\beta}},$$
     and note that, in particular, $\mathfrak{m}_\beta>n_\beta$. 

    \emph{Third change:} Fortunately,  we are left with the easiest change. Since
    $$\| \tc{z}{n_\beta}+s e_{\mathfrak{m}_\beta} - (\tc{z}{\mathfrak{m}_\beta-1}+s e_{\mathfrak{m}_\beta})\| \leq \|\tc{z}{n_\beta}- z\|=\|\tcm{z}{n_\beta}\|\rightarrow 0$$
    we may replace $\tc{z}{n_\beta}+s e_{\mathfrak{m}_\beta}$ by $\tc{z}{\mathfrak{m}_\beta-1}+s e_{\mathfrak{m}_\beta}$, which is our final goal.
 \end{proof}

 \begin{rem}
Note that the last step of the previous proof actually shows the following equality
$$ w^*-ac\{ \tc{z}{n} + se_{m}: m > n\ge 1\} = w^*-ac\{ \tc{z}{m-1} + se_{m} : m\ge 2 \},$$
for every $z\in B_{\ell_p}$ and $s>0$ such that  $\|z\|^p+s^p\le 1$. 
 \end{rem}

\subsection{\textsc{Strong boundary points}}\label{subs: sbp}

 In the description of the analytic structure of the spectrum, the strong boundary points are distinguished elements. Therefore, it is important to identify these points or at least locate them in the fibers.  Recall that $\varphi\in \M(\A_u(B_X))$ is a \emph{strong boundary point} if for every open neighborhood $U\subset \M(\A_u(B_X))$ of $\varphi$ there exists $f\in \A_u(B_X)$ satisfying $\|f\|=1=\varphi(f)$ and $|\psi(f)|<1$ for all $\psi\notin U$.

It follows directly from the definition that any strong boundary point $\varphi\in \M(\A_u(B_X))$ satisfies $\GP(\varphi)=\{\varphi\}$.

 It was established in \cite{aron1991spectra} (see also \cite{DimLasMae}) that the spectrum $\M(\A_u(B_{\ell_p}))$ contains strong boundary points in the fiber over 0. Open problem 1 in \cite{DimLasMae} raises the question as to whether the same holds for fibers over other points $z\in B_{\ell_p}$. We face here our usual first target, which is fibers over points
 $z\in B_{\ell_p}$ with  finite support and $\|z\|^p\in\Q$. In this case, we are able to prove that some of the homomorphisms defined in Proposition \ref{prop:distinta-gleason} are in fact strong boundary points. 
 
 \begin{rem} \label{rem:sbp}
     We rely on a slight modification \cite[Thm. 7.21]{stout1971theory} which is implicit in its proof: if $f\in \A_u(B_X)$  satisfies $\|f\|=\varphi(f)=1$ for some $\varphi\in \M(\A_u(B_X))$, then there exists a strong boundary point $\varphi_0\in \M(\A_u(B_X))$ with $\varphi_0(f)=1$. Specifically, \cite[Thm. 7.21]{stout1971theory} states that every $f\in \A_u(B_X)$ reaches its maximum at a strong boundary point of $\M(\A_u(B_X))$. However, its proof establishes that if $f$ takes the value 1 at an element of $\M(\A_u(B_X))$, then there is a strong boundary point where it takes the same value.  
 \end{rem}

 \begin{rem}\label{rem:automorfismo}
     By \cite[Prop. 1.6]{aron2020gleason} any automorphism $\Phi:B_{\ell_p}\to B_{\ell_p}$ such that $\Phi$ and $\Phi^{-1}$ are uniformly continuous, induces a surjective Gleason isometry $\Lambda_\Phi:\M(\A_u(B_{\ell_p}))\to \M(\A_u(B_{\ell_p}))$ given by $\Lambda_\Phi(\varphi)(f)=\varphi(f\circ\Phi)$, for all $\varphi\in \M(\A_u(B_{\ell_p}))$, $f\in \A_u(B_{\ell_p})$. Also,  if $\varphi\in \M_\omega$ then $\Lambda_\Phi(\varphi)\in \M_{\Phi(\omega)}$. Moreover, in this case it is easy to show that if $\varphi$ is a strong boundary point, then so is $\Lambda_\Phi(\varphi)$. 
 \end{rem}

 \begin{thm}\label{thm:strong boundary points}
     Let $1<p<\infty$, $m\in\N$, $z\in B_{\ell_p}$ with  $\supp( z)\subset \{1,\ldots, m\}$  and $r^p=\|z\|^p\in\Q$. Then, there are at least $\mathfrak c$ strong boundary points in $\M_z$. 
     More precisely, denoting $s=\left(1-r^p\right)^{1/p}$, for each $\xi\in\C$ with $|\xi|=1$,  there exists a strong boundary point $\varphi_{\xi}$ in $w^*-ac\{ \delta_{z+s \xi e_{n}} : n\in \mathbb{N}, n> m\}\cap \M_z$ so that  all $\varphi_{\xi}$ are different. 
  \end{thm}

 \begin{proof} We begin by showing the existence of one strong boundary point in $\M_z$. Later, we produce the family of these points described as in the statement.
     Write $r^p=\frac{l}{l+k}$ with $k\ge p$ 
     and choose $z' = (z'(1), \ldots, z'(m), 0, 0, \ldots) \in S_{\ell_{p'}}$ such that $z'(z)=\|z\|$. Take $\{a_n\}_n$ a strictly increasing sequence  of positive numbers converging to $1$,
     and consider the polynomial $P\in\mathcal P(^{l+k}\ell_p)$ given by
     $$
     P(x)=\frac{1}{r^l s^k} z'(x)^l \sum_{j>m} a_j (e_j'(x))^k.
     $$
On the one hand, proceeding as in previous proofs we see that $\|P\|=1$. On the other hand,  $|P(x)|<1$ for all $x\in \overline B_{\ell_p}$. Indeed,  $0<a_n<1$ for all $n$ and so
$$
|P(x)| = \frac{1}{r^l s^k} |z'(\pi_m(x))|^l \Bigg|\sum_{j>m} a_j (e_j'(\pi^m(x)))^k\Bigg|< \frac{1}{r^l s^k} \|\pi_m(x)\|^l \|\pi^m(x)\|^k.
$$ Noting that $\|\pi_m(x)\|^p + \|\pi^m(x)\|^p=\|x\|^p$, through Lagrange multipliers we conclude that $\frac{1}{r^l s^k} \|\pi_m(x)\|^l \|\pi^m(x)\|^k\le 1$.

Since $P(z+se_n)\underset{n\to\infty}{\to} 1$, by Remark \ref{rem:sbp} there exists a strong boundary point $\psi\in \M(\A_u(B_{\ell_p}))$  such that $\psi(P)=1$. Let us see that  $\psi$ belongs to $\M_\omega$ for some $\omega$ closely related to $z$. First, if for some $n>m$ we have  $\psi(e_n')=\lambda\not=0$, take $\varepsilon>0$ satisfying $a_n+\varepsilon<1$ and $\theta\in\R$ such that $b=\psi(z')^l \lambda^k e^{i\theta}>0$. Hence, for the  polynomial 
$$
   P_\varepsilon(x)=\frac{1}{r^l s^k} z'(x)^l \left(\sum_{j>m} a_j (e_j'(x))^k + \varepsilon e^{i\theta} e_n'(x)^k \right)
     $$ we have $\|P_\varepsilon\|=1$ and $\psi( P_\varepsilon)= \psi(P) + \frac{\varepsilon}{r^l s^k}b >1$, which is not possible. Therefore, $\psi(e_n')=0$  for all $n>m$ and then $\psi\in\M_\omega$, for some $\omega\in B_{\ell_p}$ with $\supp( \omega)\subset \{1,\ldots, m\}$.

Appealing to \cite[Prop. 3.1]{choi2021boundaries} we know that there is a net $\{x_\alpha\}_\alpha\subset S_{\ell_p}$ such that $\delta_{x_\alpha} \overset{w^*}{\to} \psi$. This implies that the net $\{x_\alpha\}_\alpha$ is weakly convergent to $\omega$. Moreover, since $\omega=\pi_m(\omega)$ and $\pi_m$ has finite rank we conclude that the net $\{\pi_m(x_\alpha)\}_\alpha$ strongly converges to $\omega$ while $\{\pi^m(x_\alpha)\}_\alpha$ weakly converges to 0. Then,
\begin{eqnarray*}
    1 &=& |\psi(P)| = \frac{1}{r^l s^k} |z'(\omega)|^l \lim_\alpha\Bigg|\sum_{j>m} a_j (e_j'(\pi^m(x_\alpha)))^k\Bigg| \le \frac{\|\omega\|^l}{r^l s^k} \lim_\alpha \|\pi^m(x_\alpha)\|^k  \\
    & = &\frac{\|\omega\|^l}{r^l s^k} \lim_\alpha (1-\|\pi_m(x_\alpha)\|^p)^{\frac{k}{p}} 
    =\frac{\|\omega\|^l (1-\|\omega\|^p)^{\frac{k}{p}} }{r^l s^k}\le 1.
\end{eqnarray*}
This is possible if and only if $\|\omega\|=r$, so we have $|z'(\omega)|=\|\omega\|=r=\|z\|=z'(z)$. By the strict convexity of $\ell_p$, we necessarily have $z=\lambda \omega$ for a certain $\lambda\in \C$ with $|\lambda|=1$.

Now, consider the automorphism $\Phi:B_{\ell_p}\to B_{\ell_p}$ given by $\Phi(x)=\lambda x$. Remark \ref{rem:automorfismo} tells us that $\Lambda_\Phi(\psi)$ is a strong boundary point which belongs to $\M_z$. Note, also,  that from $\delta_{x_\alpha} \overset{w^*}{\to} \psi$ we have $\delta_{\lambda x_\alpha} \overset{w^*}{\to} \Lambda_\Phi(\psi)$. 
Hence, renaming $\varphi=\Lambda_\Phi(\psi)$ and $y_\alpha=\lambda x_\alpha$, we get that $\varphi\in\M_z$ is a strong boundary point, $\varphi=w^*-\lim \delta_{y_\alpha}$ and $|\varphi(P)|=|\lambda^{l+k} \psi(P)|=1$. 

Our next step is to work as in Proposition \ref{prop:nocorona} to change $\{y_\alpha\}_\alpha$ by a net of the form $\{z+se_{n_\alpha}\}_\alpha$.
Recall that $\|\pi_m(y_\alpha)\|^p + \|\pi^m(y_\alpha)\|^p=\|y_\alpha\|^p=1$ for all $\alpha$ and $\{\pi_m(y_\alpha)\}_\alpha$ is strongly convergent to $z$. Then, $s=\lim_\alpha \|\pi^m(y_\alpha)\|$. Now, for each $\alpha$, define $z_\alpha=z+u_\alpha$, where
\delimiterfactor=900
\delimitershortfall=5pt
$$
u_\alpha= \left\{ \begin{array}{ccl}
            \frac{s}{
            \|\pi^m(y_\alpha)\|} \pi^m(y_\alpha)  & \quad    & \text{if } \pi^m(y_\alpha)\not=0\\
            & & \\
         0 & \quad    & \text{if } \pi^m(y_\alpha) =0.
             \end{array}
   \right.
$$     
Note that $\|z_\alpha\|\le 1$ and 
$$
\|y_\alpha - z_\alpha\|^p = \|\pi_m(y_\alpha) - z\|^p + \|\pi^m(y_\alpha)\|^p \left(1-\frac{s}{\|\pi^m(y_\alpha)\|}\right)^p \to 0.
$$
\delimiterfactor=500
\delimitershortfall=25pt
Hence, we can apply Lemma \ref{lem:redes_cerca} to deduce that $\varphi=w^*-\lim \delta_{z + u_\alpha}$.

The inequality
$$
\sum_{j>m} \Bigg|\frac{u_\alpha(j)}{s}  \Bigg|^k  \ge \frac{\big|\sum_{j>m}a_j u_\alpha(j)^k\big|}{s^k} = |P(z+u_\alpha)| \to |\varphi(P)|= 1
$$ allows us to apply Lemma \ref{lem: z casi canonico} to obtain that
$$
\|u_\alpha - s e_{\mathfrak{m}(\alpha)}\|\to 0\quad \text{where } |u_\alpha(\mathfrak{m}(\alpha))| = \max_{j}
|u_\alpha(j)|.$$
A new appeal to Lemma \ref{lem:redes_cerca} yields $\varphi=w^*-\lim \delta_{z + s e_{\mathfrak{m}(\alpha)}}$.

Finally, for each $\xi\in \C$ with $|\xi|=1$ we consider the automorphism  $\Phi_\xi:B_{\ell_p}\to B_{\ell_p}$ given by  
\delimiterfactor=900
\delimitershortfall=5pt
$$\Phi_\xi(x)(j)=\left\{ \begin{array}{ccl}
           x(j)  & \quad    & \text{if } j\le m\\
            & & \\
         \xi x(j) & \quad    & \text{if } j>m.
             \end{array}
   \right.$$
\delimiterfactor=500
\delimitershortfall=25pt   
Since $\Phi_\xi(z)=z$, by Remark \ref{rem:automorfismo} we know that $\varphi_\xi=\Lambda_{\Phi_\xi}(\varphi)$ is a strong boundary point that belongs to $\M_z$. Also, $\varphi_\xi=w^*-\lim \delta_{z + s \xi e_{c(\alpha)}}$. To see that all $\varphi_\xi$ are different, just define $Q_\ell(x)=\sum_{j>m} x(j)^\ell$ for each $\ell\ge p$ and observe that and $\varphi_\xi(Q_\ell)=s^\ell \xi^\ell$.
 \end{proof}

It is natural to conjecture that strong boundary points exist in the fibers over any $z\in B_{\ell_p}$.
Note that, as commented  in \cite{DimLasMae}, this is true for the case of $\ell_2$  due to the existence of automorphisms of the ball $B_{\ell_2}$ sending 0 to any other point.



\section{The case $p=1$}

The structure of the spectrum $\M(\A_u(B_{\ell_1}))$ differs from the previous case due to two constitutive differences between the base space $\ell_1$ and  $\ell_p$  ($1<p<\infty$). First, $\ell_1$ is not reflexive, which implies that the spectrum is fibered over $ \overline B_{\ell_1''}$ instead of $\overline B_{\ell_1}$.  Second, it lacks symmetric regularity, a highly relevant property for describing the spectrum. From \cite[Prop. 3.1]{aron2016cluster} and \cite[Thm. 4.2]{DimLasMae} we know that the only singleton fibers are those over points $z\in S_{\ell_1}$. Regarding Gleason parts, little was known about the existence of more than one Gleason part intersecting the same fiber. 
 In \cite[Cor. 4.10]{DimLasMae} it is shown that there is $\varphi\in\M_{z''}$ which does not belong to the Gleason part of $\delta_{z''}$, in the case of $z''$ being a real extreme point of $\overline B_{c_0^\perp}$.  Open problem 2 in the same article poses the question about the occurrence of different Gleason parts in fibers over arbitrary points. Here,  we complete the picture for every interior fiber, that is, those over each $z''\in B_{\ell_1''}$. Our main result in this direction is the following.

\begin{thm}\label{thm:PG-puntos-interiores-ele-1}
For every   $z''\in B_{\ell_1''}$, the fiber $\M_{z''}$ intersects $2^{\mathfrak c}$ different  Gleason parts of $\M(\A_u(B_{\ell_1}))$.
\end{thm}

Moreover, we extend \cite[Cor. 4.10]{DimLasMae} to every element in $S_{\ell_1''}$ whose projection on $c_0^\perp$ is an extreme point of the corresponding ball and some other points (see Proposition \ref{prop:enelborde} and Example \ref{ex-promedio} below).
As in the previous section, the proof of our main theorem requires several steps. Our argument deals again with a product of two homogeneous polynomials depending on different variables. Thus, we begin by computing the norm of the following prototypical polynomial.

\begin{lem}\label{lem:poli-norma-1-16}
    The polynomial $Q\in\mathcal P(^4\ell_1)$ given by $Q(x)=\sum_{i<j} (x(i)x(j))^2$ has norm $\frac{1}{16}$.
\end{lem}

\begin{proof}
 It is enough to see that any truncation $Q_N(x)=\sum_{i<j\le N} (x(i)x(j))^2$ of $Q$ attains its norm at $x_0=(\frac12,\frac12,0,0,\dots)$.  Note that $Q_N$ is invariant under permutations on the first $N$ coordinates and that $Q_N(x)=0$ for any $x$ with only one non-zero coordinate. Since we also have
 $$|Q_N(x)|\le \sum_{i<j\le N} (|x(i)| \,|x(j)|)^2 = Q_N(|x(1)|,|x(2)|,\dots), $$ 
it is enough to consider elements with non-negative coordinates having (at least) two of them non-zero. First, let us show that if $Q_N$ attains its norm at such a point, then all its non-zero coordinates must be equal. For that, due to the symmetry of $Q_N$, it suffices to prove that the first two coordinates are equal or that one of them is zero.   
We can write $x_t=(k+t,k-t, a_3, a_4, a_5,\dots) $, where $k\ge 0,\ |t|\le k$,
$a_j\ge 0$ for $j\ge 3$  and $\|x_t\|=2k+\sum_{j\ge 3} a_j=1$. If $k=0$, the first two coordinates are zero, and we are done. Thus, we assume $k>0$ and define 
$$f(t):=Q_N(x_t)= (k+t)^2 (k-t)^2 + (k+t)^2 \sum_{3\le j\le N} a_j^2 + (k-t)^2 \sum_{3\le j\le N} a_j^2  +  \sum_{3\le i<j\le N} a_i^2 a_j^2.$$ Differentiating, we get 
$$f'(t)= 4t\ \Big( t^2-k^2+ \sum_{3\le j\le N} a_j^2\Big).$$
If there is a local maximum, it is attained at  $t=0$. The other two roots of \( f' \), if they exist, must be local minima. Thus, the global maximum of \( f \) on $[-k,k]$ is attained either at zero or at the endpoints, meaning that either the first two coordinates of $x_t$ are equal or one of them is zero, as we wanted.

Now that we know that $Q_N$ attains its norm at a point where all the non-zero coordinates are equal, we observe that for $n\ge 2$ we have 
$$Q_N \Bigg( \underbrace{\frac{1}{n}, \dots, \frac{1}{n}}_{n \text{ times}}, 0, \dots \Bigg)\le {n\choose 2} \frac 1 {n^4} = \frac{n-1}{2n^3}\le \frac 1 {16}= Q_N\Bigg(\frac 1 2, \frac 1 2,0,\dots\phantom{\frac{1^1}1\hspace{-1em}}\Bigg).
$$
This completes the proof.
\end{proof}

Now, since $\ell_1''=\ell_1\oplus_1 c_0^\perp$, we can write each $z''\in B_{\ell_1''}$ as $z''=z+\omega$, with $z\in B_{\ell_1}$, $\omega\in B_{c_0^\perp}$ and $\|z''\|=\|z\| +\|\omega\|$. Our next step is to show that the elements in $c_0^\perp$ are weak-star limits of convenient nets of points from $B_{\ell_1}$.

Let us fix some notation. Recall that $\ell_1''=\ell_\infty'=C(\beta\mathbb N)'$ is described as the space of regular Borel measures on $\beta\N$. Considering this representation,  we say that an element of $\ell_1''$ (or of $c_0^\perp$) is positive if it is a positive measure. The notation $|z''|$ alludes to the measure given by the total variation of (the measure) $z''\in \ell_1''$.
Given a subset \( M \subset \mathbb{N} \), we denote by \( 1_M \) the sequence in \( \ell_\infty \) that takes the value $1$ on \( M \) and 0 elsewhere. For each \( x \in \ell_\infty \), we write \( 1_M \cdot x \) for the sequence that agrees with \( x \) on \( M \) and is zero elsewhere. We begin with a property of positive measures in \( c_0^\perp \).

\begin{lem}
   For each positive $\omega\in {c_0^\bot}$ there exists an infinite subset $M\subset \N$ such that $\omega(1_M)=0$.
\end{lem}

\begin{proof}
    We may assume $\|\omega\|=1$. We write $\N$ as a disjoint union of two infinite subsets $N_1$ and $N_2$. Since  $1 = \omega(1_{\N}) = \omega(1_{N_1}) + \omega(1_{N_2})$ and $\omega\ge 0 $  we have  that either $\omega(1_{N_1})$ or $ \omega(1_{N_2})$ is at most $\frac12$. We set $M_1=N_1$ if $\omega(1_{N_1})\le \frac12$ and $M_1=N_2$ otherwise.  We now split $M_1$ as a disjoint union of   $N^1_1$ and $N^1_2$, both infinite, and observe that
    $$\frac 1 2  \ge  \omega(1_{M_1}) = \omega(1_{N^1_1}) + \omega(1_{N^1_2}).$$ 
    So either $\omega(1_{N_1^1})$ or $ \omega(1_{N_2^1})$ is at most $\frac14$ and we set $M_2=N^1_1$ if $\omega(1_{N^1_1})\le \frac14$ and $M_2=N^1_2$ otherwise. Inductively, we get a sequence $M_1\supset M_2 \supset M_3 \supset \cdots $ satisfying $\omega(1_{M_k})\le \frac{1}{2^k}$ for every $k$. 

    Given that each $M_k$ is infinite, we can choose (inductively) a strictly increasing sequence $(m_k)_k \subset \N$ with $m_k\in M_k$ for every $k$. Let us see that the set $M= \{m_k:k\in \N\}$ satisfies  the desired property. For each $k\in \N$, using first that  $\omega$ belongs to $c_0^\bot$ and then that $\omega$ is a positive measure we get
    \begin{eqnarray*}
        \omega(1_M) &= &\omega(1_{\{m_1,\dots,m_k\}}) + \omega(1_{\{m_{k+1},m_{k+2}, \dots\}}) = \omega(1_{\{m_{k+1},m_{k+2}, \dots\}}) \\
        &\le &  \omega(1_{M_{k+1}}) \le \frac{1}{2^{k+1}}. 
    \end{eqnarray*} 
    Since this holds for every $k$, we have $\omega(1_M)=0$ as desired. 
\end{proof}

Now we are ready to show that elements in $B_{c_0^\bot}$ are weak-star limits of appropriate nets.

\begin{lem}\label{lem:red-en-M-complemento}
     For each $\omega\in B_{c_0^\bot}$, there exists an infinite subset $M\subset \N$ and a net $\{x_\beta\}_\beta \subset \|\omega\|B_{\ell_1}$ weak-star converging  to $\omega$  such that $\supp (x_\beta)\subset M^c$ for every $\beta$.
\end{lem}

\begin{proof}
    First,  observe that if $\omega$ belongs to $ B_{c_0^\bot}$, then so does $|\omega|$. The previous lemma provides us with an infinite set $M\subset \N$  such that $|\omega|(1_M)=0$. Note that for any $y\in \ell_\infty$ we have 
    $$|\omega(1_M\cdot y)| \le |\omega|(1_M\cdot |y|) \le  \|y\|_\infty |\omega|(1_M)=0.$$ 
    As a direct consequence, we get $\omega(y)=\omega(1_{M^c} \cdot y)$ for every $y\in \ell_\infty$.

  By Goldstine, take any net $\{\widetilde x_\beta\}_\beta\subset \|\omega\|B_{\ell_1}$ weak-star converging  to $\omega$ and, for each $\beta$,  define $x_\beta= 1_{M^c}\cdot \widetilde x_\beta, $ which is trivially  supported on $M^c$. To finish the proof, we notice that for every $y\in \ell_\infty$ we have
    $$
        y(x_\beta) = y ( 1_{M^c}\cdot \widetilde x_\beta ) =  ( 1_{M^c}\cdot y) ( \widetilde x_\beta ) \\
         \to \omega( 1_{M^c}\cdot y) = \omega(y). \qedhere
    $$
     
\end{proof} 

As in the previous section, on the way to our main result, we first prove that in the fiber over any interior point $z''\in B_{\ell_1''}$, there exists a homomorphism $\varphi$ belonging to a different Gleason part than $\delta_{z''}$.This result already extends \cite[Cor. 4.10]{DimLasMae} in the setting of fibers over points in the ball. For the corresponding extension of \cite[Cor. 4.10]{DimLasMae} in points on the sphere, see Proposition \ref{prop:enelborde} below.

\begin{prop}\label{prop:ele-1-morfismo-en-la-fibra}
    For every   $z''\in B_{\ell_1''}$,   the fiber $\M_{z''}$ in $\M(\A_u(B_{\ell_1}))$ contains an element $\varphi$ with $\GP(\varphi)\not=\GP(\delta_{z''})$.
\end{prop}

\begin{proof}
We write $z''=z+\omega$ with $z\in B_{\ell_1}$, $\omega\in B_{c_0^\perp}$ and $\|z''\|=\|z\|+\|\omega\|$. By the previous lemma, there is an infinite set $M\subset \N$ and a net $\{x_\beta\}_\beta \subset \|\omega\| B_{\ell_1}$  such that $\omega=w^*-\lim x_\beta$ and $\supp (x_\beta)\subset M^c$ for every $\beta$. 

Let $r=\|z''\|<1$ and $s=1-r$. By Dirichlet's approximation theorem, as in the proof of Proposition \ref{prop:ele-p-no-racional},  
      we can take sequences $(k_j), (l_j)\subset\N$ such that  
 $$\frac{l_j}{l_j+k_j}\underset{j\to\infty}{\rightarrow}r \quad\text{ and }\quad (l_j+k_j)r-l_j\underset{j\to\infty}{\rightarrow} 0.$$ Eventually by multiplying the sequences $(k_j)$ and $(l_j)$ by 4, we may assume that $k_j$ is a multiple of 4  for all $j$, condition that we will need later on to define some polynomials.

Note that $\frac{k_j}{l_j+k_j}\underset{j\to\infty}{\rightarrow}s=1-r$ and also  $(l_j+k_j)s-k_j\underset{j\to\infty}{\rightarrow} 0$. Now, writing
$
r_j=\frac{l_j}{l_j+k_j}$ and $s_j=\frac{k_j}{l_j+k_j}$, we have as in \eqref{eqn:r-s-lim-1},
\begin{equation}\label{eqn:r-s-lim-1-bis2}
 \left(\frac{r}{r_j}\right)^{l_j} \underset{j\to\infty}{\longrightarrow} 1 \quad \text{ and } \quad  \left(\frac{s}{s_j}\right)^{k_j} \underset{j\to\infty}{\longrightarrow} 1.   
\end{equation}

Given $m\in\N$, recall from  Notation \ref{not:truncados} that $\tc{z}{m} = (z(1), z(2), \ldots, z(m), 0, 0, \ldots)$ and $\tcm{z}{m}=(0,\dots, 0, z(m+1), z(m+2),\dots)$. We can take an increasing sequence of numbers $m_j\in \N$ 
satisfying
\begin{equation}\label{eqn:condicion-en-los-mj}
 \frac{L(k_j)}{s_j^{k_j}}\|\tcm{z}{m_j}\|\rightarrow 0, 
\end{equation}
where $L(k_j)$ is the Lipschitz constant from \eqref{eqn:Lipschitz}. 

Now,  we choose a sequence $(z_j')_j\subset S_{\ell_\infty}$ such that
$|z''(z'_j)|\rightarrow \|z''\|=r.$
Since $z$ is a fixed element of $\ell_1$ and $\supp(x_\beta)\subset M^c$, we may assume that 
$$\supp(z'_j)\subset \{1,\ldots, m_j\}\cup M^c.$$
Also, passing to a subsequence if necessary, we can suppose that
\begin{equation}\label{eqn:condicion-en-los-zjprima}
 \left(\frac{|z''(z'_j)|}{r}\right)^{l_j}\rightarrow 1.   
\end{equation}

Recalling that $0\in w^*-ac\{e_n-e_m:\, n<m\in M\}$ we take a net $\{e_{n_\alpha}-e_{m_\alpha}\}_{\alpha\in \Lambda}$ with $n_\alpha< m_\alpha\in M$ for all $\alpha$ such that $e_{n_\alpha}-e_{m_\alpha}\overset{w^*}{\longrightarrow} 0$. Then, if we consider the order  $(\alpha, \beta) \leq (\alpha', \beta')$ whenever $\alpha\leq\alpha'$ and $\beta\leq \beta'$, we have $y_{\alpha\beta}=\tc{z}{n_\alpha-1}+x_\beta+\frac{s}{2} (e_{n_\alpha}-e_{m_\alpha})\overset{w^*}{\longrightarrow}  z''$. Note that $\|y_{\alpha\beta}\|\le \|z\| + \|\omega\| + s=\|z''\| + s=1$, for all $\alpha$, $\beta$. Eventually taking a subnet we may assume that $\delta_{y_{\alpha\beta}}$ converges, and denote
$$ \varphi:= w^*-\lim_{(\alpha,\beta)} \delta_{y_{\alpha\beta}}\in \M_{z''}.$$
Next define, for each $j\in\N$,  polynomials $R_j\in\mathcal P(^{l_j}\ell_1)$  and $Q_j\in\mathcal P(^{k_j}\ell_1)$ by $$R_j(x) = (z'_{j}(x))^{l_j} \quad \text{and} \quad Q_j(x) = \Big(16\displaystyle\sum_{\substack{m_j<i<t \\ i, t\in M}} (e_i'(x)e_t'(x))^2\Big)^{\frac{k_j}{4}}.$$ 
By Lemma \ref{lem:poli-norma-1-16}, $\|R_j\|=\|Q_j\|=1$. Note that each $R_j$ is a  power of a linear functional, hence  $\varphi(R_j)=z''(z'_j)^{l_j}.$ On the other hand, since $Q_j$ does not depend on the first $m_j$ variables nor the variables on $M^c$,  we have
\begin{eqnarray*}
   |\delta_{y_{\alpha\beta}} (Q_j) -s^{k_j}| &=& |Q_j\left((\tc{z}{n_\alpha-1}-\tc{z}{m_j})+\frac{s}{2} (e_{n_\alpha}-e_{m_\alpha})\right)-s^{k_j} |\\
  & =& \left| Q_j\left((\tc{z}{n_\alpha-1}-\tc{z}{m_j})+\frac{s}{2} (e_{n_\alpha}-e_{m_\alpha})\right)-Q_j\left(\frac{s}{2} (e_{n_\alpha}-e_{m_\alpha})\right) \right|\\
   &\leq&  L(k_j) \|\tc{z}{n_\alpha-1}-\tc{z}{m_j}\| \\
   &\leq&  L(k_j) \|\tcm{z}{m_j}\|. \
\end{eqnarray*}
Then $|\varphi(Q_j) - s^{k_j}|\leq   L(k_j) \|\tcm{z}{m_j}\|$ and, by \eqref{eqn:condicion-en-los-mj},
\begin{equation}\label{eqn:lim-0}
\Bigg|\frac{\varphi(Q_j) - s^{k_j}}{s_j^{k_j}}\Bigg|\leq \frac{L(k_j)}{s_j^{k_j}}\|\tcm{z}{m_j}\|\underset{j\to\infty}{\longrightarrow} 0.
\end{equation}

As before, we consider the polynomial $P_j\in\mathcal P(^{l_j+k_j}\ell_1)$ given by $P_j(x)={r_j^{-l_j} s_j^{-k_j}}R_j(x)Q_j(x)$. Arguing once again that the polynomials $R_j$ and $Q_j$ depend on different variables, we know, by Lemma \ref{lem:Pinasco}, that $\|P_j\|=\|R_j\|\|Q_j\|=1$. Also, appealing  to \eqref{eqn:r-s-lim-1-bis2}, \eqref{eqn:condicion-en-los-zjprima} and~\eqref{eqn:lim-0}, we have
\begin{eqnarray*}
    |\varphi(P_j)| &=& \frac{1}{r_j^{l_j} s_j^{k_j}} |z''(z'_j)|^{l_j} |\varphi(Q_j)|\\
    &\geq& \left(\frac{|z''(z'_j)|}{r}\right)^{l_j} \left(\frac{r}{r_j}\right)^{l_j} \left( \left(\frac{s}{s_j}\right)^{k_j} -\Bigg|\frac{\varphi(Q_j) - s^{k_j}}{s_j^{k_j}}\Bigg|\right)\underset{j\to\infty}{\longrightarrow} 1.
\end{eqnarray*}
Finally, since $P_j(0)=0$ for every $j$, we conclude that $\GP(\varphi)\not=\GP(\delta_0)=\GP(\delta_{z''})$, as desired.
\end{proof}

Now we are ready to prove the main result of this section.

\begin{proof} [Proof of Theorem \ref{thm:PG-puntos-interiores-ele-1}]
The general idea of the proof is similar to that of Proposition \ref{prop:resultado-ppal-ellp}. For completeness, we provide an outline here, omitting some of the finer details. As before, we write $z''=z+\omega$ with $z\in B_{\ell_1}$, $\omega\in B_{c_0^\perp}$ and $\|z''\|=\|z\|+\|\omega\|$. We know, by Lemma \ref{lem:red-en-M-complemento}, that there is an infinite set $M\subset \N$ and a net $\{x_\beta\}_\beta \subset B_{\ell_1}$ such that $\omega=w^*-\lim x_\beta$ and $\supp (x_\beta)\subset M^c$ for every $\beta$.

Let us write $M$ as a disjoint union of infinite sets
$M = \bigcup_{u=1}^\infty M_u$.
Similar to the proof of Proposition \ref{prop:ele-1-morfismo-en-la-fibra}, for each $u$, we can take
$$\varphi^u= w^*-\lim_{(\alpha,\beta)} \delta_{\tc{z}{n_{\alpha_u-1}}+x_\beta+\frac{s}{2}(e_{n_{\alpha_u}}-e_{m_{\alpha_u}})},$$
where once again we used Notation \ref{not:truncados} and $\{e_{n_{\alpha_u}}-e_{m_{\alpha_u}}\}_{\alpha_u}$  is a net which converges weakly to 0 with $n_{\alpha_u}<m_{\alpha_u}\in M_u$. 

Let us take $m_j$, $k_j$, $l_j$ and $z'_j$ as in Proposition \ref{prop:ele-1-morfismo-en-la-fibra}, and define for each $j\in\N$ the polynomials 
$$R_j(x) = (z'_j(x))^{l_j} \quad\text{ and }\quad Q_j^u(x) = \Big(16\displaystyle\sum_{\substack{m_j<i<t \\ i, t\in M_u}} (e_i'(x)e_t'(x))^2\Big)^{\frac{k_j}{4}},$$ for  $x\in\ell_1$.  Consider
\begin{equation}\label{eq-Vsubj} V_j^u(x):=Q_j^u(x-z)=Q_j^u(x-\tcm{z}{m_j}).\end{equation}
As in Proposition \ref{prop:resultado-ppal-ellp}, we obtain that $\|V_j^u-Q_j^u \|\leq 2^{k_j} L(k_j) \|\tcm{z}{m_j}\|,$
where $L(k_j)$ is a bound for the Lipschitz constant of a norm-one $k_j$-homogeneous polynomial. Next choose $m_j$ so that
$$
\frac{2^{k_j} L(k_j)}{s_j^{k_j}}\| \tcm{z}{m_j}\|    \underset{j\to\infty}{\longrightarrow}  0.
$$

Recall from the proof of Proposition \ref{prop:ele-1-morfismo-en-la-fibra} that we have $|\varphi^u (Q_j^u) -s^{k_j} |\leq L(k_j)\|\tcm{z}{m_j}\|$. Then,
\begin{eqnarray*}
   \left|\varphi^u\left(\frac{ R_j Q_j^u}{r_j^{l_j}s_j^{k_j}}\right)\right| 
   &\geq& \left(\frac{|z''(z'_j)|}{r}\right)^{l_j} \left(\frac{r}{r_j}\right)^{l_j} \left| \left(\frac{s}{s_j}\right)^{k_j} -\frac{\varphi(Q_j^u) - s^{k_j}}{s_j^{k_j}}\right| \\
   &\geq&  \left(\frac{|z''(z'_j)|}{r}\right)^{l_j} \left(\frac{r}{r_j}\right)^{l_j} \left( \left(\frac{s}{s_j}\right)^{k_j} -\frac{ L(k_j)\|\tcm{z}{m_j}\|}{s_j^{k_j}}\right) 
\end{eqnarray*}
which shows that $\varphi^u\left(\frac{ R_j Q_j^u}{r_j^{l_j}s_j^{k_j}}\right)$ converges to $1$, uniformly in $u$. As in the proof of Proposition \ref{prop:resultado-ppal-ellp}, this fact, together with the assumption on the sequence $m_j$ implies that the numbers $t^u_j:=\varphi^u\left(\frac{ R_j V_j^u}{r_j^{l_j}s_j^{k_j}}\right)$
are uniformly bounded away from $0$ for sufficiently large $j$ and so we choose $j$ such that $|t^u_j|\geq \frac{1}{2}$ for all $ u$.

Now, for $b\in\ell_\infty$ we define the continuous polynomial 
$$ S(x)= \frac{ R_j(x) }{r_j^{l_j}s_j^{k_j}} \sum_{u=1}^\infty \frac{b(u)}{t^u_j}  V_j^u(x) .$$
Let us see that $\varphi^u(S)=b(u)$ only for $u=1$; the other cases are analogous. By definition of $t_j^1$, we have 
\begin{equation}\label{eqn:phi_V1}
    \varphi^1\left(\frac{b(1)}{t_j^1} \frac{ R_j V_j^1}{r_j^{l_j}s_j^{k_j}}\right) = \frac{b(1)}{t_j^1} t_j^1  =b(1).
\end{equation}
For the other terms of $S$, note that each $V_j^u$ depends only on the variables on $M_u \subset M$, which is disjoint from $M_1$ ($u> 1$). In particular,  $M_u$ does not contain any $n_{\alpha_1}, m_{\alpha_1}$ nor any point of $\supp(x_\beta)$ and we have 
\begin{eqnarray*}
    \Bigg|\delta_{\tc{z}{n_{\alpha_1-1}}+x_\beta+\frac{s}{2}(e_{n_{\alpha_1}}-e_{m_{\alpha_1}})}\left(\sum_{u>1}^\infty \frac{b(u)}{t^u_j} V_j^u\right)\Bigg| &=&  \Bigg|\sum_{u>1}^\infty \frac{b(u)}{t^u_j} V_j^u(\tc{z}{n_{\alpha_1-1}})\Bigg| \\
    = \Bigg|\sum_{u>1}^\infty \frac{b(u)}{t^u_j} Q_j^u(\tc{z}{n_{\alpha_1-1}}-z)\Bigg| &\leq& 2 \|b\|_\infty \left\|\sum_{u>1}^\infty Q_j^u\right\| \|\tc{z}{n_{\alpha_1-1}}-z \|^{k_j}\underset{\alpha_1\to\infty}{\longrightarrow} 0.\
\end{eqnarray*}
As a consequence, 
\begin{equation}\label{eqn:phi0env}    
\varphi^1 \left(\sum_{u>1}^\infty \frac{b(u)}{t^u_j} \frac{ R_j V_j^u}{r_j^{l_j}s_j^{k_j}} \right)=\frac{ \varphi^1(R_j)}{r_j^{l_j}s_j^{k_j}} \varphi^1 \left( \sum_{u>1}^\infty \frac{b(u)}{t^u_j} V_j^u \right)=0
\end{equation} and this, together with \eqref{eqn:phi_V1}, shows that $\varphi^1(S)=b(1)$.

Then $\{\varphi^u\}$ is an interpolating sequence, and its closure is homeomorphic to the Stone–\v{C}ech compactification of the natural numbers, $\beta(\N)$. 

To end this proof let us take $\psi_1\neq\psi_2 \in w^*-ac\{ \varphi^u : u\in \mathbb{N}\}$ and  show that $\GP(\psi_1)\not=\GP(\psi_2)$. We can find disjoint sets $\N_1\neq \N_2\subset  \N$ such that $\psi_i\in w^*-ac\{ \varphi^u : u\in \mathbb{N}_i\}$. Now let us consider the polynomials
$$S_j(x)=\frac{ R_j(x)}{r_j^{l_j}s_j^{k_j}} \sum_{v\in \N_1}^\infty Q_j^v(x) .$$
Since the polynomials  $Q^v_j$ have disjoint support and we are on $\ell_1$, it is easy to see that  
$\sum_{v\in\N_1}    Q_j^v(x)$
is continuous with norm one. By Lemma \ref{lem:Pinasco} we conclude $\|S_j\|=\|R_j\|  \left\| \sum_{v\in \N_1}^\infty Q_j^v\right\|=1$.
Thus, it is enough to show that $\psi_1(S_j)$ converges to $1$ while $\psi_2(S_j)$ does not.

We already know that $\varphi^u\left(\frac{ R_j Q_j^u}{r_j^{l_j}s_j^{k_j}}\right)\underset{j\to\infty}{\longrightarrow} 1$. Moreover, this convergence is uniform in $u$. We claim that  $\varphi^u\left(\sum_{\substack{v\in \N_1 \\ v\neq u}}^\infty \frac{ R_j Q_j^v}{r_j^{l_j}s_j^{k_j}}\right)$ converges uniformly on $u$ to $0$. To prove this, we use the fact that $\sum_{\substack{v\in \N_1 \\ v\neq u}}^\infty Q_j^u$ is  is close to $\sum_{\substack{v\in \N_1 \\ v\neq u}}^\infty V_j^u$. Given any $x\in B_{\ell_1}$, recalling \eqref{eq-Vsubj} we have
    \begin{eqnarray*}
    \Bigg|\sum_{\substack{v\in \N_1 \\ v\neq u}}^\infty V_j^v(x)- \sum_{\substack{v\in \N_1 \\ v\neq u}}^\infty Q_j^v(x) \Bigg|&=& \Bigg|\sum_{\substack{v\in \N_1 \\ v\neq u}}^\infty Q_j^v(x-\tcm{z}{m_j})-\sum_{\substack{v\in \N_1 \\ v\neq u}}^\infty  Q_j^v(x)\Bigg| \\
    &\leq & 2^{k_j}L(k_j) \| \tcm{z}{m_j}\|.\
\end{eqnarray*}

Then, since $\varphi^u\left(\sum_{\substack{v\in \N_1 \\ v\neq u}}^\infty \frac{ R_j V_j^v}{r_j^{l_j}s_j^{k_j}}\right)=0$ (see the proof of \eqref{eqn:phi0env}), we have
\begin{eqnarray*}
   \Bigg|\varphi^u\left(\sum_{\substack{v\in \N_1 \\ v\neq u}}^\infty \frac{ R_j Q_j^v}{r_j^{l_j}s_j^{k_j}}\right)\Bigg| &=& 
   \Bigg|\varphi^u\left(\sum_{\substack{v\in \N_1 \\ v\neq u}}^\infty \frac{ R_j Q_j^v}{r_j^{l_j}s_j^{k_j}}\right) - \varphi^u\left(\sum_{\substack{v\in \N_1 \\ v\neq u}}^\infty \frac{ R_j V_j^v}{r_j^{l_j}s_j^{k_j}}\right)
   \Bigg| \\
   &=& \Bigg|\varphi^u\left(\frac{ R_j }{r_j^{l_j}s_j^{k_j}}\right)\varphi^u\left( \sum_{\substack{v\in \N_1 \\ v\neq u}}^\infty Q_j^u - V_j^u \right)\Bigg| \\
   &\leq& \left(\frac{\|\tc{z}{m_j}\|}{\|z\|}\right)^{l_j}  \left(\frac{r}{r_j}\right)^{l_j} \frac{2^{k_j}L(k_j)}{s_j^{k_j}}\|\tcm{z}{m_j}\|\underset{j\to\infty}{\longrightarrow} 0. \ 
\end{eqnarray*}

Therefore, for any $u \in \N_1$ we have $\varphi^u(S_j)$ converges (uniformly on $u$) to $1$. Since $\psi_1 \in  w^*-ac\{ \varphi^u : u\in \mathbb{N}_1\}$, this implies that $\psi_1\left( S_j\right)$ also converges to $1$.

On the other hand, for any $u\in \N_2$ we have
$$ \varphi^u(S_j) =\Bigg|\varphi^u\left(\sum_{v\in \N_1}^\infty \frac{ R_j Q_j^v}{r_j^{l_j}s_j^{k_j}}\right)\Bigg|= \Bigg|\varphi^u\left(\sum_{\substack{v\in \N_1 \\ v\neq u}}^\infty \frac{ R_j Q_j^v}{r_j^{l_j}s_j^{k_j}}\right)\Bigg|,$$
which converges uniformly to $0$, hence  $\psi_2\left( S_j\right)$ converges to $0$, as we wanted to see.
\end{proof}

\medskip

Recall that the fibers in $\M(\A_u(B_{\ell_1}))$  corresponding to points \( z \in S_{\ell_1} \) are singletons (and these are the only trivial fibers). Together with our previous results, this confines the question about the existence of  
$\varphi \in \M_{z''}$  with $\GP(\varphi)\not=\GP(\delta_{z''})$  to $z''$ belonging to \( S_{\ell_1''} \setminus S_{\ell_1} \). 
In \cite[Cor. 4.10]{DimLasMae}, a positive answer is given for those $z''$ which are real extreme points of the unit ball of $c_0^\perp$. Also, note that in \cite[Sect. 4]{DimLasMae}, the arguments used to prove  that the fibers over points in \( S_{\ell_1''} \setminus S_{\ell_1} \) are nontrivial differ depending on whether or not the $c_0^\perp$-component of the given element is an extreme point. We end this section by showing that the fiber over $z''$ intersects more than one Gleason part for any $z''\in S_{\ell_1''}$ with extremal $c_0^\perp$-component.   For any ball $B$, we denote by $\Ext(B)$ the set of real extreme points of $B$.

\begin{prop}\label{prop:enelborde}
 For any   $z''\in S_{\ell_1''}$ of the form $z''=z+ \omega$, with $z\in B_{\ell_1}$ and   $\omega \in \Ext(\|\omega\| \overline{B}_{c_0^\bot})$, there is $\varphi \in \M_{z''}$ such that $\GP(\varphi)\neq \GP(\delta_{z''})$.
\end{prop}
\begin{proof}
We write $r=\|z\|$ and $s=1-r$. 
    Since \(\omega \in \Ext(s\overline{B}_{c_0^\bot})\), by the description of these extreme points done in \cite[Lem. 4.3]{DimLasMae}, 
    we know that there is a net of canonical unit vectors $(e_{n_\alpha})_{\alpha \in \Lambda}$ such that \(\omega = w^*-\lim_{\alpha \in \Lambda} s_0 e_{n_{\alpha}}\), with $|s_0|=s$.  Passing to a subnet if necessary, we  define a homomorphism 
    \(\varphi =  w^*-\lim_\alpha \delta_{\tc{z}{n_\alpha-1} + s_0e_{n_\alpha}} \in \M_{z''}\). Our objective now is to construct a sequence of norm-one polynomials \(P_j\) such that \(|\varphi(P_j) |\xrightarrow[j\to\infty]{} 1\), while \(\delta_{z''}(P_j) = \widetilde{P}_j(z'') \xrightarrow[j\to\infty]{} 0\). 

Again, as in Proposition  \ref{prop:ele-p-no-racional}, we use Dirichlet's approximation theorem to obtain sequences \((k_j), (l_j) \subset \mathbb{N}\) such that $r_j=\frac{l_j}{l_j+k_j}$ and $s_j=\frac{k_j}{l_j+k_j}$ satisfy
\begin{equation}\label{eqn:r-s-lim-bisbis}
 \left(\frac{r}{r_j}\right)^{l_j} \underset{j\to\infty}{\longrightarrow} 1 \quad \text{ and } \quad  \left(\frac{s}{s_j}\right)^{k_j} \underset{j\to\infty}{\longrightarrow} 1.   
\end{equation}
Moreover, by multiplying both \(k_j\) and \(l_j\) by 2 if necessary, we may assume that \(k_j\) is even.
Next, choose a sequence \((m_j)_j \subset \mathbb{N}\) with \(m_j \geq j\) for all \(j\) fulfilling the two conditions
\begin{equation}\label{eqn:norm-lim-bisbis}
\left(\frac{\|\tc{z}{m_j}\|}{\|z\|}\right)^{l_j} \xrightarrow[j\to\infty]{} 1  \quad \text{ and }  \quad \frac{\left|\sum_{i>m_j} z(i)^{2} + s_0^{2}\right|^{\frac{k_j}{2}}}{s^{k_j}} \xrightarrow[j\to\infty]{} 1.
\end{equation}

Take \(z'_{m_j} = (z'_{m_j}(1), \ldots, z'_{m_j}(m_j), 0, 0, \ldots) \in S_{\ell_{\infty}}\) such that \(z'_{m_j}(\tc{z}{m_j}) = \|\tc{z}{m_j}\|\) and define norm-one polynomials \(R_j \in \mathcal{P}(^{l_j}\ell_1)\) by \(R_j(x) = (z'_{m_j}(x))^{l_j}\), and \(Q_j \in \mathcal{P}(^{k_j}\ell_1)\) by \(Q_j(x) = \left(\sum_{i > m_j} (e_i'(x))^2\right)^{\frac{k_j}{2}}\).  
Now, by Lemma \ref{lem:Pinasco}, the polynomial \(P_j \in \mathcal{P}(^{l_j+k_j}\ell_1)\) given by
\(
P_j(x) = r_j^{-l_j} s_j^{-k_j} R_j(x) Q_j(x)
\)
has norm one.

Since the canonical extension to the bidual is multiplicative, \(\widetilde{P}_j = r_j^{-l_j} s_j^{-k_j} \widetilde{R}_j \widetilde{Q}_j\). We begin by computing \(\widetilde{Q}_j(z'')\):
\[
\widetilde{Q}_j(z'') = Q_j(z) + \widetilde{Q}_j(w) + 2\widecheck{\widetilde{Q}}_j(z, w).
\]
It is straightforward to verify that the second term vanishes:
\[
\widetilde{Q}_j(w) = \lim_{\alpha_1} \lim_{\alpha_2} Q(s_0 e_{n_{\alpha_1}}, s_0 e_{n_{\alpha_2}}) = 0.
\]
With a bit more work, using that \(z \in \ell_1\), we can also  show that the third term is zero as well:
\[
\widecheck{\widetilde{Q}}_j(z,w) = \frac{\widetilde{\widecheck{Q}}_j(z,w) + \widetilde{\widecheck{Q}}_j(w,z)}{2} = \widetilde{\widecheck{Q}}_j(z,w) = \lim_{\alpha} Q(z, s_0 e_{n_\alpha}) = \lim_{\alpha} (s_0 z(n_\alpha))^{k_j/2} = 0.
\]
Thus, by \eqref{eqn:r-s-lim-bisbis} and \eqref{eqn:norm-lim-bisbis} we have
\begin{eqnarray*}  
|\widetilde{P}_j(z'')| &=& \frac{|\widetilde{R}_j(z'') Q_j(z)|}{r_j^{l_j} s_j^{k_j}}= \left(\frac{\|\tc{z}{m_j}\|}{r_j}\right)^{l_j} \frac{\left|\sum_{i>m_j} z(i)^{2} \right|^{\frac{k_j}{2}}}{ s_j^{k_j}}\\
&\leq&  \left(\frac{\|\tc{z}{m_j}\|}{\|z\|}\right)^{l_j} \left(\frac{r}{r_j}\right)^{l_j} \left(\frac{\|\tcm{z}{m_j}\|}{ s_j}\right)^{k_j}\xrightarrow[j \to \infty]{} 0.
\end{eqnarray*}

Now, for \(n_\alpha > m_j\), we compute:
\[
\delta_{\tc{z}{n_\alpha - 1} + s_0 e_{n_\alpha}}(R_j) = (z_{m_j}'(\tc{z}{m_j}))^{l_j} = \|\tc{z}{m_j}\|^{l_j},
\]
and
\[
\delta_{\tc{z}{n_\alpha - 1} + s_0 e_{n_\alpha}}(Q_j) = \left( \sum_{i=m_j+1}^{n_\alpha} z(i)^2 + s_0^2 \right)^{\frac{k_j}{2}}.
\]

Combining these expressions with \eqref{eqn:r-s-lim-bisbis} and \eqref{eqn:norm-lim-bisbis}, we obtain the desired result:
\begin{align*}
|\varphi(P_j)| &= \frac{\|\tc{z}{m_j}\|^{l_j}}{r_j^{l_j} s_j^{k_j}} \Bigg| \sum_{i > m_j} z(i)^2 + s_0^2 \Bigg|^{\frac{k_j}{2}} \\
&= \left( \frac{\|\tc{z}{m_j}\|}{\|z\|} \right)^{l_j} \left( \frac{r}{r_j} \right)^{l_j} \left( \frac{s}{s_j} \right)^{k_j} \frac{\left| \sum_{i > m_j} z(i)^2 + s_0^2 \right|^{\frac{k_j}{2}}}{s^{k_j}} \xrightarrow[j \to \infty]{} 1. \qedhere
\end{align*}

\end{proof}

We do not know if the previous result extends to  $z''=z+\omega$ with an arbitrary $\omega $. However, we can see that the conclusion of Proposition \ref{prop:enelborde} holds whenever  $\omega$ is a multiple of a convex combination of extreme points of $\overline{B}_{c_0^\bot}$ with disjoint support. For the sake of simplicity, we show this only for a particular $z''$. The general fact follows combining these ideas with those of Proposition \ref{prop:enelborde}.
\begin{exa}\label{ex-promedio}
    Consider  \( z \in B_{\ell_1} \), with \( \|z\| = \frac{1}{2} \) and \( \supp(z) \subset \{1, \ldots, m\} \) and take $\omega$ a weak-star limit of $\{ (e_{2n} + e_{2n+1})/4 \}$. Then, for $z''=z+\omega$,  there exists $\varphi \in \M_{z''}$ such that $\GP(\varphi)\neq \GP(\delta_{z''})$.
\end{exa}

To see this, note that we can  chose a net $\{n_\alpha\}_\alpha$ such that \( \omega = \frac{\omega_1 + \omega_2}{2} \) with
\[
\omega_1 = w^*-\lim_{\alpha} \frac{e_{2n_{\alpha}}}{2}, \quad \omega_2 = w^*-\lim_{\alpha} \frac{e_{2n_{\alpha}+1}}{2}.
\]

Now, eventually passing to a subnet, we can define
\[
\varphi = w^*-\lim_\alpha \delta_{z + \tfrac{1}{4}(e_{2n_\alpha} + e_{2n_\alpha+1})},
\]
and consider the polynomials
\[
R(x) = (z'(x))^2, \quad Q(x) = \sum_{i > m} \left( e_{2i}'(x) + e_{2i+1}'(x) \right)^2,
\]
where \( z' \in S_{\ell_\infty} \) is chosen so that \( z'(z) = \|z\| \) and \( \supp(z') \subset \{1, \ldots, m\} \).  
Define \( P(x) = 16\,R(x)\,Q(x) \)  and note that $\|P\|=1$. It is not hard to verify that for \( n_\alpha > m \),
\[
\delta_{z + \tfrac{1}{4}(e_{2n_\alpha} + e_{2n_\alpha+1})}(P)
= 16 R(z) Q\left( \tfrac{1}{4}(e_{2n_\alpha} + e_{2n_\alpha+1}) \right) = 1,
\]
implying \( \varphi(P) = 1 \).

On the other hand, we have
\[
\delta_{z''}(Q) = \widetilde{Q}(\omega) = \lim_{\alpha_1} \lim_{\alpha_2}\widecheck Q\left( \tfrac{e_{2n_{\alpha_1}} + e_{2n_{\alpha_1}+1}}{4}, \tfrac{e_{2n_{\alpha_2}} + e_{2n_{\alpha_2}+1}}{4} \right) = 0,
\]
which produces \( \delta_{z''}(P) = 16 \delta_{z''}(R) \delta_{z''}(Q) = 0 \), and thus \( \GP(\varphi) \neq \GP(\delta_{z''}) \).

\section*{Acknowledgments}

The authors would like to thank the anonymous referee for a careful reading of the manuscript and for valuable comments and suggestions that helped improve the paper. 

{The first-named author was partially supported by CONICET PIP 11220200102366CO  and UBACyT
20020220300242BA. The second and third-named authors were partially supported by CONICET PIP 11220200101609CO. }

\bibliography{biblio.bib}
\bibliographystyle{plain}
\end{document}